\documentclass[a4paper,11pt]{article}
\usepackage[T1]{fontenc}
\usepackage[utf8]{inputenc}
\usepackage[english]{babel}
\usepackage[margin=1.1in]{geometry}
\usepackage{microtype}
\usepackage{braket}
\usepackage{amsmath}
\usepackage{amssymb}
\usepackage{amsthm, aliascnt}
\usepackage{mathtools}
\usepackage{mathrsfs}
\usepackage{xcolor}
\usepackage[hyphens]{url}
\usepackage{hyperref}
\usepackage{enumitem}
\usepackage{tikz}
\usepackage[maxbibnames=99]{biblatex}
\usepackage{csquotes}
\usetikzlibrary{patterns}
\usepackage{tikz}
\usepackage{esint}
\usepackage{mleftright} %
\usepackage{cases} %
\usepackage{cancel}
\usepackage{cleveref}

\usepackage{array}   %
\newcolumntype{L}{>{$}l<{$}} %
\newcolumntype{C}{>{$}c<{$}} %
\newcolumntype{R}{>{$}r<{$}} %
\usepackage{booktabs} %

\hypersetup{
                colorlinks=true,
                linkcolor={magenta},
                citecolor={cyan},
                breaklinks=true}

\theoremstyle{plain}
\newtheorem{thm}{Theorem}
\numberwithin{thm}{section}
\numberwithin{equation}{section}
\theoremstyle{definition}
\newaliascnt{defi}{thm}
\newtheorem{defi}[defi]{Definition}
\aliascntresetthe{defi}

\theoremstyle{plain}
\newaliascnt{lemma}{thm}%
\newtheorem{lemma}[lemma]{Lemma}
\aliascntresetthe{lemma}
\theoremstyle{plain}
\newaliascnt{prop}{thm}%
\newtheorem{prop}[prop]{Proposition}
\aliascntresetthe{prop}
\theoremstyle{plain}
\newaliascnt{cor}{thm}%
\newtheorem{cor}[cor]{Corollary}
\aliascntresetthe{cor}
\theoremstyle{definition}
\newaliascnt{ex}{thm}%

\aliascntresetthe{ex}
\theoremstyle{definition}
\newaliascnt{rem}{thm}%
\newtheorem{rem}[rem]{Remark}
\aliascntresetthe{rem}
\theoremstyle{plain}

\DeclarePairedDelimiter{\abs}{\lvert}{\rvert}
\DeclarePairedDelimiter{\norm}{\lVert}{\rVert}

\DeclareMathOperator{\tr}{tr}

\DeclareMathOperator{\jac}{Jac}

\DeclareMathOperator{\diag}{diag}
\DeclareMathOperator{\essinf}{essinf}
\DeclareMathOperator{\argmax}{arg\, max}
\newcommand{\R}{\mathbb{R}}
\newcommand{\M}{\mathcal{M}}
\renewcommand{\S}{\mathbb{S}}

\newcommand{\N}{\mathbb{N}}
\newcommand{\ssubset}{\subset\joinrel\subset}

\newcommand{\Hn}{\mathcal{H}^{n-1}}

\newcommand{\eps}{\varepsilon}

\DeclareMathOperator{\divv}{div}
\DeclareMathOperator{\supp}{supp}

\DeclareMathOperator{\sgn}{sgn}

\DeclareMathOperator{\id}{Id}

\newcommand{\measurerestr}{\raisebox{-.250ex}{$\mathbin{\vrule height 2ex depth 0pt width 0.13ex\vrule height 0.13ex depth 0pt width .7ex}$}}
\newcommand\mrestr[2]{
  #1 \,\measurerestr_{\,#2} %
  }

\newcommand\restr[2]{{%
  \mleft.\kern-\nulldelimiterspace %
  #1 %
  \littletaller %
  \mright|_{#2} %
  }}

\newcommand{\littletaller}{\mathchoice{\vphantom{\big|}}{}{}{}}

\makeatletter
\newcommand{\leqnomode}{\tagsleft@true\let\veqno\@@leqno}
\newcommand{\reqnomode}{\tagsleft@false\let\veqno\@@eqno}

\newenvironment{proofItemize}{\begin{itemize}[align=right,itemindent=1em,labelsep=.5em,labelwidth=1em,leftmargin=0pt,nosep]}{\end{itemize}}

\newcommand{\om}{\omega}
\newcommand{\Om}{\Omega}
\newcommand{\al}{\alpha}
\newcommand{\la}{\lambda}
\newcommand{\La}{\Lambda}
\newcommand{\red}[1]{#1}
\newcommand{\de}{\delta}
\newcommand{\De}{\Delta}
\newcommand{\na}{\nabla}
\newcommand{\be}{\beta}

\newcommand{\Ga}{\Ga}
\newcommand{\vp}{\varphi}
\newcommand{\tht}{\theta}
\newcommand{\ze}{\zeta}
\newcommand{\pa}{\partial}
\newcommand{\tx}{\tilde{x}}
\newcommand{\fr}[2]{\frac{#1}{#2}}

\newcommand{\sbs}{\subset}
\newcommand{\sbse}{\subseteq}
\newcommand{\Spn}{\mathbb{S}^{n-1}}

\newcommand{\sm}{\setminus}
\newcommand{\li}{\liminf}

\newcommand{\ml}{\mleft}
\newcommand{\mr}{\mright}
\newcommand{\sh}{\sharp}

\newcommand{\hu}{\widehat{u}}
\newcommand{\hw}{\widehat{w}}

\newcommand{\hg}{\widehat{g}}
\newcommand{\hnu}{\widehat{\nu}}

\newcommand{\hup}{\widehat{u'_t}}
\newcommand{\hwp}{\widehat{w'_t}}

\newcommand{\F}{\mathcal{F}}
\newcommand{\G}{\mathcal{G}}
\newcommand{\Ll}{\mathcal{L}}
\newcommand{\J}{\mathcal{J}}

\newcommand{\D}{\mathcal{D}}

\newcommand{\tphi}{{\widetilde{\Phi}}}
\newcommand{\tPhi}{\tphi}
\newcommand{\tEde}{\widetilde{E}_\de}

\makeatother

\title{Sharp quantitative Talenti inequality in particular cases}
\author{P. Acampora, J. Lamboley}
\date{}

\begin{document}
\reversemarginpar
\maketitle

\begin{abstract}
In this paper, we focus on the famous Talenti's symmetrization inequality, more precisely its $L^p$ corollary asserting that the $L^p$-norm of the solution to $-\Delta v=f^\sharp$ is higher than the $L^p$-norm of the solution to $-\Delta u=f$ (we are considering Dirichlet boundary conditions, and $f^\sharp$ denotes the Schwarz symmetrization of $f:\Omega\to\R_+$). We focus on the particular case where functions $f$ are defined on the unit ball, and are characteristic functions of a subset of this unit ball. We show in this case that stability occurs for the $L^p$-Talenti inequality with the sharp exponent 2.
\end{abstract}

\section{Introduction}\label{sect:intro}

In this paper, we will investigate some stability versions of Talenti's inequality in particular cases. Given $\Om\subset \R^n$ a bounded open set and $f\in L^2(\Om)$, we will denote by $u_f$ the unique weak solution to
\begin{equation}\label{eq:elliptic}
\left\{\begin{array}{ccccc}
-\Delta u_f&=&f&\textrm{ in }&\Om,\\
u_f&=&0&\textrm{ on }&\partial\Om.\end{array}\right.
\end{equation}
The aim of Talenti's inequality is to compare $u_f$ and $u_{f^\sharp}$ where $f^\sharp:\Om^\sharp\to\R$ is the Schwarz symmetrization of $f$, defined on $\Om^\sharp$ the centered ball of same volume as $\Om$ (see Definition \ref{defi: schwarz}): more precisely Talenti's inequality (Theorem \ref{thm: talenti}) states that if $f\geq 0$, then
$$\forall x\in \Om^\sharp, \;\;u_f^\sharp(x)\leq v(x),$$
where $v=u_{f^\sharp}$ solves
\begin{equation}\label{eq:ellipticf*}
\left\{\begin{array}{ccccc}
-\Delta v&=&f^\sharp&\textrm{ in }&\Om^\sharp\\
v&=&0&\textrm{ on }&\partial\Om^\sharp.\end{array}\right.
\end{equation}
and $u_f^\sharp$ is the Schwarz symmetrization of $u_f$. This \red{pointwise} comparison implies in particular the following: for every $p\in[1,+\infty]$,
\begin{equation}\label{eq:Lptalenti}
\|u_f\|_{L^p(\Om)}\leq \|v\|_{L^p(\Om^\sharp)}.
\end{equation}
Moreover, \red{equality is realized in \eqref{eq:Lptalenti} for some $p\in[1,+\infty]$ only if $\Om$ is a ball and $f=f^\sharp$ up to translations:
this was proved in \cite[Theorem 1 and Corollary 1]{ALT86} (we recall this in Theorem \ref{thm: talentirigid} below, the reader can also look at \cite{K06})}.\\

In this work we are interested in quantitative versions of \eqref{eq:Lptalenti}. Notice that there are two parameters that are symmetrized in these inequalities: $\Om$ and $f$. Therefore the stability inequalities we may be seeking for should take into account both the distance from $f$ to $f^\sharp$ and from $\Om$ to $\Om^\sharp$: to evaluate the asymmetry of $\Om$ for example, we denote
$$\alpha(\Om)=\min_{x_0\in\R^n}\left\{\frac{|\Om\Delta B_r(x_0)|}{|B_r|}, \;|B_r|=|\Om|\right\}$$
the Fraenkel asymmetry of $\Om$ ($B_r(x)$ denotes the ball of radius $r>0$ and centered at $x\in\R^n$; we also denote $B_r=B_r(0)$ the centered ball of radius $r$).

As far as we know, there are only two partial results in this direction:
\begin{enumerate}
\item in \cite{K21}, the author focuses on the case $f\equiv 1$ and shows that for every $p\in[1,+\infty]$, there exists $c=c(n,p)$ such that for every $\Om$ bounded open set in $\R^n$, 
\[
\begin{split}
	&\|v\|^p_{L^p(\Om^\sharp)}-\|u_1\|^p_{L^p(\Om)}\geq c\alpha(\Om)^{2+p}\qquad \textrm{ if }p\in[1,+\infty),\\ &\|v\|_{L^\infty(\Om^\sharp)}-\|u_1\|_{L^\infty(\Om)}\geq c\alpha(\Om)^{3}.
\end{split}
\]
\item in \red{\cite[Theorem 1.1]{ABMP24}} the authors \red{solve} the case $p=\infty$: it is shown that for any $\Om$ bounded open set of $\R^n$ and any nonnegative $f\in L^2(\Om)$, there exists $c=c(n,|\Om|,f^\sharp)$ and $\theta=\theta(n)$ such that
\begin{equation}\label{eq:stabinfty}\|v-u^\sharp\|_{L^\infty(\Om^\sharp)}\ge c\left(\alpha(\Om)^3+\inf_{x_0\in\R^n}\Big\|f-f^\sharp(\cdot+x_0)\Big\|_{L^1(\R^n)}^\theta\right).
\end{equation}
Note that the proof of Talenti's inequality shows that $v-u^\sh$ is radially decreasing, so that $\norm{v-u^\sh}_{L^\infty(\Om^\sharp)}=\norm{v}_{L^\infty(\Om^\sharp)}-\norm{u}_{L^\infty(\Om)}$.

In the same paper \red{(see \cite[Section 7]{ABMP24})}, the authors also obtain a result in the case \red{$p<\infty$}, namely:
\red{$$\|v\|_{L^p(\Om^\sharp)}^p-\|u\|_{L^p(\Om)}^p\ge c\alpha(\Om)^{2+p},$$}
\red{but this result is only partial as there is no control of the asymmetry of $f$.}
\end{enumerate}

In both of these results, the obtained exponents are not expected to be sharp (also in \red{\eqref{eq:stabinfty}}, $\theta$ does not seem to be explicit).

The goal of this paper is to provide a first quantitative result for \eqref{eq:Lptalenti} with a sharp exponent. In order to start this investigation with only one parameter (similarly to \cite{K21} where the author assumed $f\equiv 1$), we will assume that $\Om=B_1$ is a centered ball of radius 1, so  that $\Om=\Om^\sharp$ and the only parameter is $f$. Moreover, we will assume $f=\chi_E$ to be the characteristic function of a set $E\subset B_1$: in this case we denote $u_E$ for $u_{\chi_E}$, and $f^\sharp=\chi_{B_*}$ where $B_*$ is the centered ball of same volume as $E$. This will allow us to use a geometric approach and the framework of shape derivatives. In this setting, we obtain the following stability result:
\begin{thm}
    \label{thm: pnorm}
    Let $n\ge 1$, $p\in [1,+\infty]$, and $m\in(0,|B_1|)$. Then there exists $c=c(p,m,n)>0$ such that for every measurable set $E\subset B_1$ with $\abs{E}=m$ we have
    \begin{equation}\label{eq: stabpnorm}
        \norm{u_{B_*}}_{L^p(B_1)}-\norm{u_E}_{L^p(B_1)}\ge {c}\abs{E\De B_*}^2,
    \end{equation}
where $B_*$ is the centered ball of volume $m$.\end{thm}
\red{\begin{rem}
Note that the constant $c$ in \eqref{eq: stabpnorm} is explicit only when $p=1$ where we have $c=m/16$ if $n=1$ and $c=(4n^2\omega_n)^{-1}$ if $n\geq 2$ (it is therefore uniform in $m$ if $n\ge2$), see the proof of Proposition \ref{prop: p=1}. But in the case $p>1$, our approach argues by contradiction and classically this cannot provide an explicit estimate of the constant.
\end{rem}}
\noindent Moreover, we show in Propositions \ref{prop: p=1}, \ref{prop: sharp} and \ref{prop: p=infty} that the exponent 2 obtained in Theorem \ref{thm: pnorm} is sharp. In a forthcoming \red{work we} will generalize the strategy introduced in the present paper to investigate the more general case where $f$ is not assumed to be a characteristic function.\\

The proof of \autoref{thm: pnorm} will be divided in three cases depending on the value of $p$. In the special case $p=1$, the proof is quite straightforward and relies only on Talenti's inequality: we provide the proof in Section \ref{ssect: p=1}, which is based on the study of the following shape optimization problem:
$$
\max\Set{ \norm{u_E}_1 | \begin{gathered}
	|E|=m,\\|E\Delta B_*|=\delta
\end{gathered}}
$$
for $m\in(0,|B_1|)$ and $\delta>0$, which can be seen as looking for the worst set of given asymmetry, with regard of the Talenti deficit $\|u_{B_*}\|_1-\|u_E\|_1$. We identify explicitely the solution of this problem as a  target-shaped set (it is the union of a ball and an annulus, see \Cref{defi: annulus}).\\

In the case $p>1$, we were not able to solve explicitly  $\max\{\|u_E\|_p, |E|=m, |E\Delta B_*|=\delta\}$, we therefore follow a strategy developed in \cite{M21,MRB22,CMP23,CMP23supp} to prove \eqref{eq: stabpnorm}. In fact the case $1<p<\infty$ falls into the following framework: given a function $j:\R_+\to\R$, we study the stability for the following shape optimization problem:
\begin{equation}
\max\Set{\J(E)| \abs{E}=m},\qquad \qquad \J(E):=\int_{B_1} j(u_E(x))dx 
\end{equation}
which is solved by the centered ball $B_*$ if $j$ is non-decreasing.
Classically, under convexity assumption on $j$, such problem can be relaxed in the class
\[
\M_m:=\Set{V\in L^\infty(B_1) | \begin{gathered}
	0\leq V \leq 1, \\ \int_{B_1} V=m
\end{gathered}}
\]
where $m\in(0,|B_1|)$, and we define $\J(V):=\int_{B_1}j(u_V)$  for $V\in L^\infty(B_1)$ non-negative. More precisely we prove:
\begin{thm}
    \label{thm: maintheorem}
    Let $n\ge 1$, $m\in(0,\abs{B_1})$, and  $j\in C^1(\R_+)\cap C^2((0,+\infty))$ be   such that $j'(0)\ge 0$ and $j''(s)>0$ for every $s>0$, and such that
    \begin{equation}
    \label{eq: j''summable}
       \red{\lim_{s\to0^+}s^\al j''(s)\in (0,+\infty)}
    \end{equation}
    \red{for some $\al\in(-\infty,1)$}. Then there exists a positive constant $c=c(j,m,n)$ such that for every $V\in\M_m$ we haveB
    \[
        \J(B_*)-\J(V)\ge c\norm{V-\chi_{B_*}}_1^2.
    \]
    where $B_*$ is the centered ball of volume $m$.
\end{thm}

\begin{rem}\label{rem: holder}
\red{In this result, we introduce the suitable assumptions on $j$ in order to include the case $j(s)=|s|^p$ for $p\in(1,\infty)$. Those hypotheses may be devided in two categories:
\begin{enumerate}
\item that $j$ is strictly increasing and strictly convex, which is used in most stages of the proof, see for example Proposition \ref{prop: exist}.
\item the technical assumption \eqref{eq: j''summable}  about the behavior of $j$ near the origin, that will allow the computation that we need in the technical part of the proof, see for example Proposition \ref{prop: w'}. Note that this assumption implies that $j'$ is locally $\be:=(1-\al)$-H\"older continuous. This follows easily from \red{observing that $j''(s)\le Cs^{-\al}$ and} the equality $j'(y)-j'(x)=\int_x^y j''(t)dt$.
\end{enumerate}
}
\end{rem}

\begin{proof}[Proof of \Cref{thm: pnorm} from \Cref{thm: maintheorem} when $p\in(1,\infty)$:]
    Let $p\in(1,\infty)$: we consider $j:s\in[0,\infty)\mapsto s^p$, which is $C^2$ in $(0,\infty)$ and $C^1$ up to 0, with $j'(0)=0$ and $j''>0$ on $(0,\infty)$, and satisfies \eqref{eq: j''summable}.

   Therefore, \Cref{thm: maintheorem} applies, and there exists $c=c(p,m,n)$ such that
    \[
        \norm{u_{B_*}}_p^p\ml(1-\ml(\fr{\norm{u_E}_p}{\norm{u_{B_*}}_p}\mr)^p\mr)\ge c \abs{E\De B_*}^2.
    \]
    Then the result follows by using the inequality
    \[
        1-x^p\le p(1-x) \qquad \qquad \forall x\ge 0.
    \]
\end{proof}

Finally, we give a seperate proof of \autoref{thm: pnorm} when $p=\infty$, see below for more details.

\subsection*{Outline of the paper} 
In the following section, we recall basic definitions on Schwarz symmetrization, and then produce a proof of \Cref{thm: pnorm} in the particular case $p=1$. We then provide useful results to prepare the proof of \autoref{thm: maintheorem}: in particular we prove that  \red{if $j$ is strictly increasing and stricly convex, then}
\[
\J(B_*)=\max\Set{\J(V):=\int_{B_1} j(u_V(x))dx| V\in \M_m},
\]
that $B_*$ is the unique maximizer, and that for any small $\delta>0$ there exists $E_\delta$ such that $|E_\delta|=m$, $|E_\delta\Delta B_* |=\delta$ and 
\[ 
\J(E_\delta)=\max\Set{\J(V)| \begin{gathered}
	V\in \M_m, \\
    \norm{V-\chi_{B_*}}_1=\delta 
\end{gathered}}.
\]
We also show that the exponent 2 in the stability result of \Cref{thm: pnorm} and \Cref{thm: maintheorem} is sharp.

We then focus on the proof of \Cref{thm: maintheorem}, which is split in the next two sections: 
in \Cref{fugledeImpliesStability}, we show that the proof of \Cref{thm: maintheorem} reduces to a Fuglede type of result (\Cref{thm: fuglede}) asserting that stability occurs for smooth deformations of $B_*$. This strategy falls into the strategy of ``selection principle'' first introduced in the setting of shape optimization by Cicalese and Leonardi \cite{CL12} to prove the quantitative isoperimetric inequality: the main tool here is the quantitative baththub principle proved in \cite{MRB22}.
\Cref{sect:fuglede}, the most technical part of the paper, is devoted to the proof of \Cref{thm: fuglede}. It requires a computation of first and second order shape derivative of the energy (\Cref{ssect:shapederivatives}), a proof of the coercivity of the second order optimality condition (\Cref{ssect:coercivity}), and a continuity property of the second order derivative to control the Taylor expansion (\Cref{ssect:continuity}).

In \Cref{sect:Linf} we provide the proof of \Cref{thm: pnorm} for $p=+\infty$. The proof uses a representation of the solution with the Green function of the ball, and relies on adapting the shape derivative approach of \Cref{fugledeImpliesStability} and \Cref{sect:fuglede} to the $L^\infty$ norm. This requires to compute a second order shape derivative of such norm, which up to our knowledge, is new (see \cite{HenLucPhil2018} for a computation of a first order shape derivative in a similar context).

Finally, we detailed in the appendix some results \red{needed in \Cref{fugledeImpliesStability}} about the way convergence of functions may imply the convergence of level sets. \red{The proofs are straightforward, but we have decided to include the computations to complement the results as stated in~\cite{M21,MRB22} (see the appendix for more details).}

\noindent{\bf Acknowledgement: }the authors would like to warmly thank Idriss Mazari, for several fruitful discussions about this paper, which in particular lead to a simplification of the proof of \autoref{thm: maintheorem} in \autoref{ssect: mainproof}. This work was also partially supported by the project ANR STOIQUES financed 
by the French Agence Nationale de la Recherche (ANR). This work was partially supported by Gruppo Nazionale per l’Analisi Matematica, la Probabilità e le loro Applicazioni
(GNAMPA) of Istituto Nazionale di Alta Matematica (INdAM). 
The author Paolo Acampora was partially supported by the 2025 project GNAMPA 2024: "Analisi di Problemi Inversi nelle Equazioni alle Derivate Parziali", CUP\_E5324001950001 .

\section{Tools and first results}
\subsection{Symmetrization and Talenti's inequality
}
  We refer the reader to \cite{T94} or \cite{K06} for an overview about definitions and properties of rearrangements.
        \begin{defi}\label{def: mu}
	       Let $\Omega\subset\R^n$ be a bounded open set, and let $u: \Omega \to \R$ be a measurable function. We define the \emph{distribution function} $\mu_u : [0,+\infty)\, \to [0, +\infty)$ of $u$ as the function
	       $$
	           \mu_u(t)= \abs{\Set{x \in \Omega \, ,\,  \abs{u(x)} > t}}.
	       $$
        \end{defi}
        \begin{defi}\label{AAC:defi:incre}
        	Let $u: \Omega \to \R$ be a measurable function. We define the \emph{decreasing rearrangement} $u^*$ of $u$ as 
         \[
            u^*(s) = \inf\Set{t>0 \, ,\; \mu_u(t) \leq s}.
         \] 
        
        \end{defi}
        \begin{defi}\label{defi: schwarz}
            Let $u: \Omega \to \R$ be a measurable function. We define the \emph{Schwarz rearrangement} $u^\sharp$ of $u$ as
            \[
            u^\sharp (x)= u^*(\omega_n \abs{x}^n) \qquad \qquad x\in \Om^\sharp,
            \]
            where $\Om^\sh$ denotes the centered ball having the same volume as $\Om$.
        \end{defi}
        The next result is the famous symmetrization result obtained by G. Talenti, see \cite[Theorem I]{T76}.
        \begin{thm}[Talenti's comparison]
        \label{thm: talenti}
        Let $\Omega\subset\R^n$ be a bounded open set, let $f\in L^{2}(\Om)$
        and let $v_f\in {H^1_0}(\Om)$ and $v_{f^\sh}\in {H^1_0}(\Om^\sh)$ be the unique solutions to
        \[
            \left\{\begin{array}{cccc}
            -\De u_f&=&f  & \text{ in }\Om, \\
            u_f&=&0 &\text{ on }\pa \Om,          
            \end{array}\right. \qquad \qquad
            \left\{\begin{array}{cccc}
             -\De u_{f^\sh}&=&f^\sh &\text{ in }\Om^\sh, \\
            u_{f^\sh}&=&0 &\text{ on }\pa \Om^\sh.           
            \end{array}\right.
        \]
        Then
        \[
            u_f^\sh\le u_{f^\sh}.
        \]
        In particular, for every set $E$, letting $u_E=u_{\chi_E}$ and $u_{E^\sh}=u_{\chi_{E^\sh}}$, we have
        \[
            u_E^\sh\le u_{E^\sh}.
        \]
        \end{thm}
        The following result solves the case of equality in Talenti's inequality and was first proved in \cite{ALT86}.
        \begin{thm}
        \label{thm: talentirigid}
            Let $\Om\subset\R^n$ be a bounded and open set\red{, let $p\in[1,+\infty]$, and let $f\in L^2(\Om)$ be a non-negative function such that $f\not\equiv 0$}. If 
            \red{
            \[
            \norm{u_f}_p=\norm{u_{f^\sh}}_p,
            \]
            }
            then there exists $x_0\in\R^n$ such that up to negligible sets $\Om=x_0+\Om^\sh$, $f(\cdot)=f^\sh(\cdot+x_0)$, and $u_f(\cdot)=u_{f^\sh}(\cdot+x_0)$ almost everywhere.
        \end{thm}
        \red{
        \begin{proof}
            If $p<+\infty$, we have that by Talenti's inequality
            \[
                \norm{u_f}_p^p=\int_{\Om}(u_f^\sh)^p\,dx \le \int_{\Om^\sh}(u_{f^\sh})^p\,dx =\norm{u_{f^\sh}}_p^p.
            \]
            In particular, if we have equality, by strict monotonicity of the function $s\mapsto s^p$, we have that $u_f^\sharp=u_{f^\sharp}$ almost everywhere, and by~\cite[Theorem 1]{ALT86} we have the result. 
            
            For the case $p=+\infty$ we refer to~\cite[Corollary 1]{ALT86}.
        \end{proof}
        }

\subsection[Proof of main Theorem for p=1]{Proof of Theorem \ref{thm: pnorm} for $p=1$}\label{ssect: p=1}

In the case $p=1$, the proof of Theorem \ref{thm: pnorm} is quick and relies only of Talenti's symmetrization result. It also shows the importance of the following annulus.
\begin{defi}\label{defi: annulus}
    Let $m\in(0,\abs{B_1})$. For every $\delta\in(0,\min\{2m,2(|B_1|-m)\})$ we define the \emph{optimal $\delta$-asymmetric radial open set} $A_\delta$ as
    \[
        A_\delta=B_{r_1(\delta)}\cup\Set{x\in\R^n | r_*<\abs{x}< r_2(\delta)},
    \]
    where $r_*$ is the radius of $B_*$ the ball of volume $m$, and $r_1(\delta)$ and $r_2(\delta)$ are chosen in such a way that 
    \[
    \abs{A_\delta\Delta B_*}=\delta.
    \]
    Namely, we have
    \[
        r_1(\delta)=\frac{1}{\abs{B_1}^{1/n}}\mleft(m-\frac{\delta}{2}\mright)^{\frac{1}{n}}
    \qquad\textrm{ and }\qquad
        r_2(\delta)=\fr{1}{\abs{B_1}^{1/n}}\mleft(m+\frac{\delta}{2}\mright)^{\frac{1}{n}}.
    \]
\end{defi} 
The next result shows that among sets $E$ of given volume and such that $|E\Delta B_*|=\delta$, $A_\delta$ is the worst set with regard to the deficit $\|u_{B_*}\|_1-\|u_E\|_1$.

\begin{prop}
\label{prop: f=s}
Let $j:s\in\R_+\mapsto s$, i.e. for every measurable set $E\subset B_1$,
    \[
        \J(E)=\int_{B_1}u_E\,dx.
    \]
Let $m\in(0,|B_1|)$ and {$\delta\in(0,\inf\{2m,2(|B_1|-m)\})$}. Let $E\subset B_1$ be a measurable set such that $|E|=m$ and $|E\Delta B_*|=\delta$ where $B_*$ is the centered ball of volume $m$. Then
    \[
        \J(E)\le \J(A_\de),\;\; i.e.\;\;\;\int_{B_1}u_E\,dx\leq \int_{B_1}u_{A_\delta}\,dx.
    \]
where $A_\delta$ is defined in Definition \ref{defi: annulus}.
\end{prop}
\begin{proof}
    By linearity we may write
    \[
        u_E=u_{E\cap B_*}+u_{E\cup B_*}-u_{B_*}.
    \]
    Using Talenti's comparison result (see \Cref{thm: talenti}), we know that 
    \[
        u_{E\cap B_*}^\sh \le u_{B_{r_1(\delta)}}, \qquad \qquad u_{E\cup B_*}^\sh\le u_{B_{r_2(\delta)}}.
    \]
    Integrating over $B_1$ the two inequalities, and using the equi-measurability of the Schwarz rearrangement, we get:
   \[
    \begin{split}
    \int_{B_1}u_E\,dx &=\int_{B_1}u_{E\cap B_*}^\sh\,dx+\int_{B_1}u_{E\cup B_*}^\sh\,dx-\int_{B_1}u_{B_*}\,dx\\
            &\le \int_{B_1}u_{B_{r_1(\delta)}}\, dx+\int_{B_1}u_{B_{r_2(\delta)}}\, dx-\int_{B_1}u_{B_*}\,dx =\int_{B_1}u_{A_\de}\,dx%
        .
    \end{split}
    \]
\end{proof} 
 
\begin{prop}\label{prop: p=1}
With the same notations as in \Cref{prop: f=s}, there exists a positive constant $c=c(m,n)$ such that
\[
    \J(B_*)-\J(E)\ge c\abs{E\De B_*}^2,
\]
for every measurable set $E\subseteq B_1$ such that $\abs{E}=m$. Moreover, the exponent 2 is optimal, in the sense that the inequality cannot be valid for any lower exponent.
\end{prop}
\begin{proof}
Let $E\subset B_1$ of volume $m$. Then $\delta=|E\Delta B_*|\in[0,
\min\{2m,2(|B_1|-m)\})$. If $\delta=0$ then $E=B_*$ a.e. and $\J(B_*)=\J(E)$.
If however $\delta>0$, then we have by Proposition \ref{prop: f=s}:
$$\J(B_*)-\J(E)\geq \J(B_*)-\J(A_\delta).$$
We introduce $w\in H^1_0(B_1)$ the unique solution to $-\Delta w=1$ in $B_1$. Classical computations lead to $w(x)=w(\abs{x})=\frac{1-|x|^2}{2n}$. Moreover
\begin{eqnarray*}\J(B_*)-\J(A_\delta)&=&\int_{B_1}(-\Delta w)[u_{B_*}-u_{A_\delta}]
=\int_{B_1}w[\chi_{B_*}-\chi_{A_\delta}]\\
&=&|\S^{n-1}|\left(\int_{r_1(\delta)}^{r_*}w(r)r^{n-1}dr-\int_{r_*}^{r_2(\delta)}w(r)r^{n-1}dr\right)\\
&=&\frac{|\S^{n-1}|}{2n(n+2)}\Big(r_1(\delta)^{n+2}+r_2(\delta)^{n+2}-2r_*^{n+2}\Big)
\end{eqnarray*}
When $n=1$ we compute explicitely $r_1(\delta)^3+r_2(\delta)^3-2r_*^3=\frac{3}{2^4}m\delta^2$  so that $c=m/2^4$ works. In the case $n\geq 2$, we use that the function $\varphi:x\in(0,1)\mapsto x^{(n+2)/n}$ is strongly convex ($\varphi''\geq \frac{2(n+2)}{n^2}$), so that
\begin{equation*}
\begin{split}	
r_1(\delta)^{n+2}+r_2(\delta)^{n+2}-2r_*^{n+2}&=\varphi\left(\frac{m-\frac{\delta}{2}}{|B_1|}\right)+\varphi\left(\frac{m+\frac{\delta}{2}}{|B_1|}\right)-2\varphi\left(\frac{m}{|B_1|}\right)\\[7 pt]
&\geq\frac{2(n+2)}{n^2}\left(\frac{\delta}{2|B_1|}\right)^2
\end{split}
\end{equation*}
hence the result with $c=\frac{1}{4n^2\omega_n}$.
Finally, to show that the exponent is optimal, we use a Taylor expansion in the previous computation to obtain:
\[\J(B_*)-\J(A_\delta)=\frac{m^{2/n-1}}{4n^2\omega_n^{2/n}}\delta^2+o(\delta^2).\]
Therefore, if $\alpha\in(0,2)$ then
\[\frac{\J(B_*)-\J(A_\delta)}{|B_*\Delta A_\delta|^\alpha}\xrightarrow[\;\delta\to 0\;]\; 0.\]
\end{proof}
\begin{rem}
    \label{rem: j'positive}
    We notice that the case $p=1$ implies \Cref{thm: maintheorem} when $j$ is in $C^1(\R_+)$, convex and $j'(0)>0$. Indeed, we have  by convexity of $j$
    \[
        \J(B_*)-\J(E)\ge \int_{B_1}j'(u_{E})(u_{B_*}-u_E)\,dx \ge j'(0) \ml(\norm{u_{B_*}}_1-\norm{u_E}_1\mr),
    \]
    so the stability inequality for $\J$ follows from \Cref{prop: p=1}.
   Of course, this does not apply to the case $j:s\mapsto s^p$ ($p>1$) whose derivative vanishes at 0.
\end{rem}
\subsection{Optimization among densities}\label{ssect: densities}

In this section, we prove the following:

\begin{prop}
\label{prop: exist}
Let $j\in C^0(\R_+)$ be a convex non-decreasing function and $m\in(0,|B_1|)$. Then the functional $\J:V\in \M_m\mapsto \int_{B_1}j(u_V)$ is well-defined, convex, and
\begin{equation}\label{eq:maxJ}
            \J(B_*)=\max_{V\in\M_m}\J(V).
\end{equation}
Moreover, the following properties hold:
\begin{enumerate}
\item for any $\delta\in(0,\min\{2m,2(|B_1|-m)\})$, we define the set of \emph{fixed asymmetry} weights
\begin{equation*}
\mathcal \M_m^\de:=\Set{V \in \M_m\; ,\; \norm{V-\chi_{B_*}}_{1}=\delta }.
\end{equation*}
Then the optimization problem \begin{equation}
\label{Eq:PvDelta}
\sup_{V \in \M_m^\de}\J(V)
\end{equation} admits a bang-bang solution, which means there exists a set $E_\de$ such that 
\[
\J(E_\de)=\max_{V \in{\M_m^\de}}\J(V).
\]
\item if $j$ is strictly convex and strictly increasing in $(0,\infty)$, then $\chi_{B_*}$ is the unique maximizer to~\eqref{eq:maxJ}.
\end{enumerate}
\end{prop}

   Before proving this result, we state first a classical elliptic regularity theorem that will be used often throughout the paper. We refer for instance to \cite[Theorem 9.13, Theorem 8.16]{GT98}.
    \begin{thm}
    \label{thm: regularity}
        Let $q>\max\{1,n/2\}$. Then there exists $C=C(n,q)$ such that for every $f\in L^q(B_1)$,
        \[
            \norm{u_f}_{2,q}\le C\norm{f}_q,
        \]
        where $u_f{\in H^1_0(B_1)}$ is the unique solution to \eqref{eq:elliptic}.
    \end{thm}
Also, we need a few technical results: 
classically, weak-$*$ $L^\infty$ convergence does not imply convergence of the $L^1$-norm, but if one has some control on the sign of the involved functions, then one can retrieve such convergence:
\begin{lemma}
\label{lem: weakstarcomp}
Let $h\in L^\infty(B_1)$ and $(h_k)$ be a sequence of functions in $L^\infty(B_1)$ such that
\[
h_k\xrightharpoonup{*-L^\infty(B_1)} h.
\]
We assume that there exists $E\sbs B_1$ such that 
\[\forall k\in\N, \qquad
h_k\ge 0 \text{ in $E$}, \qquad\textrm{ and }\qquad  h_k\le 0 \text{ in $B_1\setminus E$}.
\]
Then
\begin{equation}
    \label{eq: L1conv}
\lim_k \int_{B_1} \abs{h_k}\, dx=\int_{B_1} \abs{h}.
\end{equation}
\end{lemma}
\begin{proof}
It is sufficient to notice that under these assumptions, 
\[
h_k^+=h_k\chi_{E}, \qquad \qquad h_k^-=h_k\chi_{E^c},
\]
where $h_k^+$ and $h_k^-$ are the positive part and the negative part of $h_k$ respectively. In particular, by weak convergence of $h_k$, we get
\[
\begin{split}
h_k^+\xrightharpoonup{*-L^\infty({B_1})} h^+=h\chi_{E}, \\[7 pt]
h_k^-\xrightharpoonup{*-L^\infty({B_1})} h^-=h\chi_{E^c},
\end{split}
\]
which implies \eqref{eq: L1conv}    
\end{proof}

\begin{lemma}\label{lem: step1comp}
Let $m\in(0,|B_1|)$ and $\delta\in(0,\min\{2m,2(|B_1|-m)\})$. Then the sets $\M_m, \M_m^\de$ are convex and compact with respect to the weak-$*$ $L^\infty$ topology, and their extremals are characteristic functions.
\end{lemma}

\begin{proof}
\begin{proofItemize}
\item 
The convexity for $\M_m$ is immediate by definition, while for $\M_m^\de$, we need to show that if $W_0,W_1\in\M_m^\de$ and 
    \[
        W_\alpha=\alpha W_1+(1-\alpha)W_0,
    \]
for $\al\in(0,1)$, then $\norm{W_\alpha-V_0}_1=\delta$. This is true since $0\le W_\alpha\le 1$, and by definition of $V_0$ we have
    \[
        W_\al\le V_0 \text{ in }B_*, \qquad \qquad   W_\al\ge V_0 \text{ in }B_1\setminus B_*,
    \]
    so that explicitly computing the $L^1$ norms,
    \[
        \norm{W_\al-V_0}_1=\alpha\norm{W_1-V_0}_1+(1-\alpha)\norm{W_0-V_0}_1=\de.
    \]
    \item $\M_m$ is compact because if $W_k$ weakly-$*$ converges in $L^\infty$ to some $W$, then 
    \[
        m=\lim_k \int_{B_1}W_k\,dx=\int_{B_1}W\,dx.
    \]
    For what regards $\M_m^\de$, if we take a sequence $W_k\in \M_m^\de$ converging to $W$ weakly-$*$ in $L^\infty$, then the functions $\chi_{B_*}-W_k$ satisfy the assumptions of \Cref{lem: weakstarcomp}, and then
    \[
        \de=\lim_{k}\norm{W_k-\chi_{B_*}}_1=\norm{W-\chi_{B_*}}.
    \]
    \item The fact that extremal points of $\M_m$ are characteristic functions is classical (see for instance \cite[Prop. 7.2.17]{HP18}). Let us detail the same result for $\M_m^\de$:  let $W\in \M_m^\de$, and assume that
    \[
        \abs{\Set{0<W<1}}>0 ;
    \]
    we must show that $W$ is not extremal. Since $\norm{W}_1=m=\|\chi_{B_*}\|_1$, we have
    \[
        0<\de=\norm{W-\chi_{B_*}}_1=2\int_{B_*}^{}(1-W)\,dx \le 2\abs{\set{0<W<1}\cap B_*}.
    \]
    Analogously we prove $\abs{\set{0<W<1}\setminus B_*}>0$. Therefore, for small $\eps>0$ we can split the set $\set{\eps <W<1-\eps}$ in four pairwise disjoint subsets $S_1,S_2,S_3,S_4$ of $B_1$ such that 
    \[
    \abs{S_1}=\abs{S_2}>0, \qquad \qquad \abs{S_3}=\abs{S_4}>0,
    \]
    and
    \[
        S_1\cup S_2 = \Set{\eps<W<1-\eps}\cap B_* \qquad \qquad S_3\cup S_4 = \Set{\eps<W<1-\eps}\sm B_*.
    \]
    Under these assumptions we can write
    \[
        W=\fr{1}{2}(W-\eps\chi_{S_1\cup S_3}+\eps\chi_{S_2\cup S_4})+\fr{1}{2}(W+\eps\chi_{S_1\cup S_3}-\eps\chi_{S_2\cup S_4}),
    \]
    with $W\mp\eps\chi_{S_1\cup S_3}\pm\eps\chi_{S_2\cup S_4}\in \M_m^\de$. This proves that a non-bang-bang function is not an extreme point for $\M_m^\de$. 
    \end{proofItemize}
\end{proof}

\begin{proof}[Proof of Proposition \ref{prop: exist}] For $V\in\M_m$, $u_V\geq 0$ is bounded, so $j(u_V)\in L^1(B_1)$ and $\J$ is well-defined.

\medskip\begin{proofItemize}
	
\item {\bf Convexity of $\J$:}   given $W_0,W_1\in \M_m$ and $\alpha\in[0,1]$, if we take $W_\alpha=\alpha W_1+(1-\alpha)W_0$, then by linearity with respect to $V$ of the equation \eqref{eq:elliptic}, we get
    \[
        u_{W_\alpha}=\alpha u_{W_1}+(1-\alpha)u_{W_0},
    \]
    which, joint with the convexity of $j$ gives
    \[
        \J(W_\alpha)\le \alpha \J(W_1)+(1-\alpha)\J(W_0).
    \]

\item {\bf $\J(\cdot)$ is continuous with respect to the weak-$*$ $L^\infty$ convergence:} let $W_k$ be a sequence converging in the weak-$*$ $L^\infty$  sense to some function $W_\infty$. Let $u_k:=u_{W_k}\in H^1_0(B_1)$ solution to \eqref{eq:elliptic}.
    We show that $u_k$ converges to $u_{W_\infty}$.
    Let $q>n$; since $(W_k)$ are equibounded in $L^\infty$, \Cref{thm: regularity} applies, and we have that $(u_k)$ are equi-bounded in $W^{2,q}({B_1})$, namely there exists some constant $C=C(n,q)>0$ such that
    \[
        \norm{u_k}_{2,q}\le C.
    \]
    Therefore, by Kondrachov Theorem, there exists a subsequence (not relabelled) such that
    \[
        u_k\xrightarrow[]{W^{1,q}({B_1})}u
    \]
    for some $u\in W^{{1},q}({B_1})$. In particular, by Sobolev's imbeddings, the convergence happens strongly in $L^\infty$, and $u\in H^1_0(B_1)$. Moreover, using the weak-$*$ convergence of $W_k$, we get, passing to the limit in the weak formulation of  \eqref{eq:elliptic}, 
    \[
        \int_{B_1} \nabla u\cdot\nabla \varphi \, dx=\int_{B_1} W_\infty\,\varphi \, dx \qquad \qquad \forall\varphi\in H^1_0({B_1}),
    \]
    which implies $u=u_{W_\infty}$. Finally, since $j$ is continuous and $u_k$ converges strongly in $L^\infty$ to $u$, then
    \[
        \lim_k \J(W_k)=\lim_k \int_{B_1}j(u_k)\,dx=\J(W_\infty).
    \]
    Since the argument is valid for every choice of the subsequence, this proves the continuity.

 \item {\bf Solution to \eqref{eq:maxJ} and existence to \eqref{Eq:PvDelta}:} by \Cref{lem: step1comp} about compactness of $\M_m$ and $\M_m^\delta$ and the continuity proved in the previous step, we get the existence of solutions  to \eqref{eq:maxJ} and \eqref{Eq:PvDelta}. As extremal points of $\M_m$ and $\M_m^\delta$ are bang-bang functions, by convexity of $\J$ we get that there exist bang-bang solutions to \eqref{eq:maxJ} and \eqref{Eq:PvDelta}. Moreover, the monotonicity of $j$ implies that $j(u^\sh)=j(u)^\sh$, so that, using Talenti's inequality, we deduce that $B_*$ solves \eqref{eq:maxJ}.

\item {\bf Uniqueness:} since $j$ is strictly convex, also $\J$ is strictly convex and any maximizer of $\J$ is necessarily bang-bang. In particular, if $E$ is an optimal set such that $\chi_E$ maximizes $\J$ we have that 
        \begin{equation}
        \label{eq: equalityJ}
            \J(B_*)=\J(E)=\int_{B_1}j(u_E)\,dx =\int_{B_1}j(u_E^\sh)\,dx,
        \end{equation}
        where we used that the monotonicity of $j$ implies $j(u_E^\sh)=j(u_E)^\sh$.
        By Talenti's inequality we know that $j(u_E^\sh) \le j(u_{B_*})$. The strict monotonicity of $j$ and \eqref{eq: equalityJ} ensure that $u_E^\sh=u_{B_*}$ almost everywhere. By the rigidity of Talenti's inequality (see \Cref{thm: talentirigid}) we get $E=B_*$.
\end{proofItemize}
        
    \end{proof}

\subsection{Sharpness of the exponent}
We show that the exponent 2 is sharp in \Cref{thm: maintheorem}, in the sense that for $\alpha\in(0,2)$, one can find a sequence of sets $E_k\in\M_m\setminus\{B_*\}$ such that
\[\frac{\J(B_*)-\J(E_k)}{|B_*\Delta E_k|^\alpha}\xrightarrow[\;k\to\infty\;]\; 0.\]
We proceed as in \Cref{ssect: p=1} by using the annulus $A_\delta$ (who is a candidate as being the worst asymmetric set with regard to Talenti's deficit, though we are not in position to prove it except in the case of the $L^1$-norm). More precisely, we show the following result: 

\begin{prop}
\label{prop: sharp}
    Let $m\in(0,\abs{B_1})$, and let $j\in C^1(\R_+)$ be a non-constant, non-decreasing, convex function. Then there exists $\bar{\de}>0$ and a positive constant $C=C(j,m,n)$ such that for every $\delta\in (0,\bar{\de})$,  
    \begin{equation}
        \label{eq: sharpnessADelta}
        \fr{1}{C}\de^2\le \J(B_*)-\J(A_\delta)\le C\delta^2.
    \end{equation}
\end{prop}

For the proof of \Cref{prop: sharp} and in several other places in this paper, we need to define the notion of adjoint state (see also \cite[Section 5.8]{HP18}):

\begin{defi}[Adjoint State]
\label{defi: adjoint}
    Let $j\in C^1(\R_+)$. For every $V\in L^2(B_1)$ nonnegative, we define the \emph{adjoint state} $w_V$ of $u_V$ as the unique function solving in the weak sense the \emph{adjoint problem}
    \[
        \begin{dcases}
            -\Delta w_V=j'(u_V) &\text{in }{B_1},\\[5 pt]
            w=0 &\text{on }\partial{B_1},
        \end{dcases}
    \]
    namely,
    \begin{equation}
    \label{eq: adjoint}
        \int_{{B_1}} \nabla w_V\cdot\nabla \varphi \, dx= \int_{B_1} j'(u_V)\varphi \qquad\qquad \forall\varphi\in H^1_0({B_1}).
    \end{equation}
    When $V=\chi_E$ we will write $w_E:=w_{\chi_E}$.
\end{defi}

\begin{proof}[Proof of \Cref{prop: sharp}]
    For every $\de \in [\,0,\min\{2m,\,2|B_1|-2m\}\,)$, let us define $u_\de=u_{A_\de}$, $u_0=u_{B_*}$, and $w_\de =w_{A_\de}$, $w_0=w_{B_*}$. Moreover, since these functions are all radial, we will identify, with a slight abuse of notation
    \[
    u_{\de}(x)=u_{\de}(\abs{x}), \qquad \qquad w_{\de}(x)=w_{\de}(\abs{x}).
    \]
    By convexity of $j$, we have
    \[
      \int_{B_1}j'(u_\de)(u_0-u_\de)\,dx\le \J(B_*)-\J(A_\de)\le\int_{B_1}j'(u_0)(u_0-u_\de)\,dx.
    \]
    After two integration by parts, we may rewrite
    \begin{equation}
      \label{eq: upperAndLowerJADelta}
      \int_{B_1}w_\de (\chi_{B_*}-\chi_{A_\de})\,dx\le \J(B_*)-\J(A_\de)\le\int_{B_1}w_0 (\chi_{B_*}-\chi_{A_\de})\,dx.
    \end{equation}
  Noticing that $\chi_{B_*}-\chi_{A_\de}$ is non-zero only between the radii $r_1(\de)$ and $r_2(\de)$, we focus our attention on the values of $w_\de$ and $w_0$ near $\abs{x}=r_*$, where we recall that $B_*=B_{r_*}$.
  \begin{proofItemize}
	
  \medskip
  \item \textbf{Estimates for $w_0$:}  
    by Taylor expansion,
      \begin{equation}
        \label{eq: taylorw0}
            w_{0}(r)=w_{0}(r_*)+\pa_r w_{0}(r_*)(r-r_*)+R_1(r),
        \end{equation}
        with 
        \[
          R_1(r) = \int_{r_*}^r \ml(\pa_r w_0(s)-\pa_r w_0(r_*)\mr)\,dr.
          \]
          Since $w_0$ solves the equation $-\De w_0 = j'(u_0)$, then by classical elliptic regularity (\Cref{thm: regularity}) we have that $w_0\in C^{1,\be}$ for some $\be \in (0,1)$, so that
          \[
            \abs{\pa_r w_0(r_*)-\pa_r w_0(s)} \le C\abs{s-r_*}^\be,
            \]
            and 
            \begin{equation}
              \label{eq: remainderw0}
              \abs{R_1(r)}\le C \abs{r-r_*}^{1+\be}.
            \end{equation}
            Therefore, noticing that (see \Cref{defi: annulus} for the definition of $r_1=r_1(\delta), r_2=r_2(\delta)$)
            \[
              \int_{B_1}(\chi_{B_*}-\chi_{A_\de})\,dx = 0 \qquad \qquad \chi_{B_*}(r)-\chi_{A_\de}(r)=\sgn(r_*-r)\chi_{(r_1,r_2)}(r),
              \]
              we obtain by \eqref{eq: taylorw0} and \eqref{eq: remainderw0} 
              \[
                \int_{B_1}w_0(\chi_{B_*}-\chi_{A_\de})\, dx =\pa_r w_0(r_*)|\S^{n-1}|\int_{r_1}^{r_2} r^{n-1}\abs{r-r_*}\,dr + R_2(\de)
                \]
                with
                \[
                  \abs{R_2(\de)} \le C \abs{r_2-r_1}^{2+\be}.
                  \]
                  In particular, since $r_2-r_1=O(\de)$, we have 
                  \begin{equation}
                    \label{eq: upperEstimateJADelta}
                    \int_{B_1}w_0(\chi_{B_*}-\chi_{A_\de})\, dx \le C\de^2 + o(\de^2),
                  \end{equation}
                  so that, with \eqref{eq: upperAndLowerJADelta}, the upper bound in \eqref{eq: sharpnessADelta} is proven for a suitable choice of $\bar{\de}$.
                  
                  \medskip
                  \item \textbf{Estimates for $w_\de$:} we now need to adapt the previous estimate with $w_\delta$ instead of $w_0$: to that end, we first notice that $w_\de$  smoothly converges to $w_0$ for $\de$ that goes to 0. Indeed, since $-\De u_\de = \chi_{A_\de}$, we have from classical elliptic regularity (see \Cref{thm: regularity}) that 
                  \[
                    u_\de \xrightarrow{L^\infty(B_1)} u_0.
                    \]
                    In particular, the continuity of $j'$ ensures that $j'(u_\de)$ converges in $L^\infty$ to $j'(u_0)$ and, for any $\be\in(0,1)$,
                    \[
                      w_\de \xrightarrow{C^{1,\be}(B_1)} w_0.
                      \]
                      This convergence, joint with \eqref{eq: taylorw0} and \eqref{eq: remainderw0} ensures that for some uniform positive constant $C$ we have
                  \begin{equation}
                    \label{eq: taylorwDelta}
                        w_{\de}(r)=w_{\de}(r_*)+\pa_r w_{\de}(r_*)(r-r_*)+R_3(r),
                    \end{equation}
                    with 
                    \begin{equation}
                      \label{eq: remainderWDelta}
          \abs{R_3(r)} \le C \abs{r-r_*}^{1+\be}.
          \end{equation}
          If $\bar{\de}$ is small enough, since $j'\ge 0$  and $j'\not\equiv0$ we know by Hopf's lemma that $w_0$ is strictly radially decreasing and we have
          \begin{equation}
            \label{eq: estimateDerwDelta}
            -\pa_r w_\de(r_*) \ge -\fr{1}{2}\pa_r w_0(r_*)>0.
          \end{equation}
          We now compute similarly to \eqref{eq: upperEstimateJADelta}, using \eqref{eq: taylorwDelta}, \eqref{eq: remainderWDelta}, and \eqref{eq: estimateDerwDelta} we obtain
          \[
            \int_{B_1}w_\de (\chi_{B_*}-\chi_{A_\de})\, dx \ge c \de^2 + o(\de^2),
          \]
          for $c>0$ small enough, thus concluding the proof.
\end{proofItemize}
    \end{proof}
\begin{rem}
We notice that the convexity assumption in \Cref{prop: sharp} is only needed to make the proof simpler. As in \cite[Proof of Formula (32)]{M21} we could have used a parametric derivative approach to obtain the same result without assuming $j$ to be convex.
\end{rem}
In the following we show the exponent 2 is sharp also for $p=\infty$.
\begin{prop}\label{prop: p=infty}
    Let $m\in(0,\abs{B_1})$. Then there exists $\bar{\de}>0$ and a positive constant $C=C(m,n)$ such that for every $\delta\in (0,\bar{\de})$,  
    \begin{equation}
        \label{eq: sharpnessADeltaInfty}
        \fr{1}{C}\de^2\le \norm{u_{B_*}}_\infty-\norm{u_{A_\de}}_\infty\le C\delta^2.
    \end{equation}
\end{prop}
\begin{proof}
    Since both $u_{B_*}$ and $u_{A_\de}$ are radially decreasing, if $G$ denotes the Green's function on the ball (see \autoref{ssect: fugledeinfinity} for the definition), we have
    \[
        \norm{u_{B_*}}_\infty-\norm{u_{A_\de}}_\infty = \int_{B_1}^{}G(0,y)(\chi_{B_*}-\chi_{A_\de})\,dy.
    \]
    Using a Taylor expansion of $G(0,\cdot)$, the fact that $\abs{\na_y G(0,\cdot)}>0$ near $\pa B_*$, and using the volume constraint, we obtain the result as in the proof of \Cref{prop: sharp}. 
\end{proof}

\section{Proof of Theorem \ref{thm: maintheorem}}
\label{fugledeImpliesStability}

In this section, we provide the first part of the proof of Theorem \ref{thm: maintheorem} (which implies Theorem \ref{thm: pnorm} in the case $p\in(1,\infty)$ as shown in \Cref{sect:intro}). In fact, in the spirit of the Selection principle used in \cite{CL12} to deduce the sharp quantitative isoperimetric inequality from a stability result by Fuglede (\cite{F89}), we show that \Cref{thm: maintheorem} is a consequence of the following \Cref{thm: fuglede}, asserting stability among smooth deformations of $B_*$. In what follows, for every set $E\subset B_1$ and every vector field $\Phi: B_1 \to \R^n$, we use the notation
\[
    E^\Phi = (\id +\Phi)(E).
\]

\begin{thm}
    \label{thm: fuglede}
    Let $m\in(0,\abs{B_1})$ and $j\in C^1(\R_+)\cap C^2((0,\infty))$ be such that $j'(0)\geq 0$, $j''(s)>0$ for every $s>0$ and
    \begin{equation}
    \label{eq: j''summablebis}
          \red{\lim_{s\to0^+}s^\al j''(s)\in (0,+\infty)}
    \end{equation}
    \red{for some $\al\in(-\infty,1)$}. Then there exist positive constants $c=c(j,m,n)$, $\eta=\eta(j,m,n)$ such that for every $\Phi\in W^{2,q}({B_1},\R^n)$ orthogonal to $\pa B_*$ with
    \[
        \norm{\Phi}_{W^{2,q}}\le \eta, \qquad \qquad \abs{B_*^\Phi}=\abs{B_*}=m,
    \]
    we have
    \[
         \J(B_*)-\J(B_*^{\Phi})\ge c \abs{B_*\De B_*^{\Phi}}^2.
    \]
\end{thm}
\red{Precisely, we show in this section} how one can deduce \Cref{thm: maintheorem} from \Cref{thm: fuglede} in 2 subsequent steps\red{:
\begin{enumerate}
\item using the results from subsection \ref{ssect: densities} we first prove that it is sufficient to show stability in an $L^1$-neighborhood of $B_*$,
\item then we use a {\it selection principle}: assuming that the stability inequality is false for a sequence of sets $E_\delta$, then we replace those sets by smoother sets $\widetilde{E_\delta}$ for which the stability inequality will also be false, then contradicting Theorem \ref{thm: fuglede}. More precisely, the sets $\widetilde{E_\delta}$ will be chosen as level sets of the adjoint state, following ideas from the works \cite{M21,MRB22,CMP23}.
\end{enumerate}
}

The proof of Theorem \ref{thm: fuglede} is postponed to \Cref{sect:fuglede}. 
Let us stress that the orthogonality assumption in \Cref{thm: fuglede} can be removed, since every deformation $\Phi$ can be replaced with a deformation $\Psi$ normal to $\pa B_*$ in a way that $B_*^\Phi=B_*^\Psi$ and $\norm{\Psi}_{{W^{2,q}}}$ is equivalent to $\norm{\Phi}_{{W^{2,q}}}$ (see for instance \cite[Lemma 5.9.5]{HP18}).

\subsection{Step 1: local stability implies global stability}
In this first step, we show that it is enough to prove \Cref{thm: maintheorem} in the regime $\|V-\chi_{B_*}\|_1\to 0$: for $m\in(0,|B_1|)$ given  and $\delta$ small enough, we recall
\[
\mathcal \M_m^\de:=\Set{V \in L^\infty(B_1) | \begin{gathered}
	0\le V\le 1, \\ \int_{B_1} V=m, \\ \norm{V-\chi_{B_*}}_{1}=\delta 
\end{gathered}}.
\]
\begin{prop}
\label{prop: smalldelta}
   Let $j\in C^0(\R_+)$ be strictly convex and strictly increasing. If
    \begin{equation}
        \label{eq: smallasymm}
        \liminf_{\delta\to 0}\inf_{V\in\M_m^\delta}\frac{\J(B_*)-\J(V)}{\delta^2}>0,
    \end{equation}
    then \Cref{thm: maintheorem} holds true.
\end{prop}

\begin{proof}
    We denote 
    \[
    c=\liminf_{\delta\to 0}\inf_{V\in\M_m^\delta}\frac{\J(B_*)-\J(V)}{\delta^2}
    \]
    that is assumed to be positive, and we let $V_k$ be a minimizing sequence for the problem
    \[
        \inf_{\substack{V\in \M_m \\ V\neq \chi_{B_*}}} G(V):=\frac{\J(B_*)-\J(V)}{\|V-\chi_{B_*}\|_1^2}.
    \]
    By compactness of $\M_m$ (see \Cref{lem: step1comp}) we may assume that $V_k$ weakly-$*$ $L^\infty$ converges to some function $V_\infty\in \M_m$. 
    
    First case: if $V_\infty=\chi_{B_*}$, then by weak-* convergence and the sign constraints on $V_k-\chi_{B_*}$ we can apply \Cref{lem: weakstarcomp} and  we get:
    \[
        \lim_k\norm{V_k-\chi_{B_*}}_1=0,
    \]
    and so by definition of $c$,
    $$\lim_kG(V_k)\geq c>0.$$

    Second case: if $V_\infty\neq \chi_{B_*}$ then again by \Cref{lem: weakstarcomp} $\lim_k\|V_k-\chi_{B_*}\|_1=\|V_\infty-\chi_{B_*}\|_1$ and by continuity of $\J$ (see \Cref{lem: step1comp}),
    \[
        \lim_kG(V_k)=G(V_\infty).
    \]
    But by the uniqueness stated in Proposition \ref{prop: exist}, $G(V_\infty)>0$, hence the result.
\end{proof}

   \subsection{Step {2}: \Cref{thm: fuglede} implies \Cref{thm: maintheorem}}\label{ssect: mainproof}

To conclude the proof of \Cref{thm: maintheorem}, we will use the following result from \cite{MRB22}: 
\begin{thm}[Quantitative bathtub principle]
\label{eq: quantbathtub}
    Let $\Om\subseteq\R^n$ be an open bounded set,  $m\in(0,\abs{\Om})$, and $u\in C^{1,\al}(\Om)$ for some $\al\in(0,1)$. We assume 
    \[\abs{\Set{u>t_*}}=\abs{\Set{u\geq t_*}}=m\qquad
    \textrm{ and }
\quad \min_{\partial\{u>t_*\}}|\nabla u|>0\]
    for some unique $t_*\in\R$. Then there exists a positive constant $c=c(\norm{u}_{C^{1,\al}})$ such that for every $V\in L^1(\Om)$ non-negative with
    $\int_\Om V\,dx=m$,
    \[
        \int_\Om uV\le \int_\Om u\chi_{\{u>t_*\}}-c\left(\min_{{\partial\{u>t_*\}}}\abs{\na u}\right)\fr{{\Hn(\partial B_*)}}{\Hn({\partial \{u>t_*\}})}\norm{V-\chi_{\{u>t_*\}}}_1^2,
    \]
    where $B_*$ is the ball of volume $m$.
\end{thm}

    \begin{proof}[Proof of \Cref{thm: maintheorem} from Theorem \ref{thm: fuglede}]        
    
        By \Cref{prop: exist} we have that for every $\de>0$ small there exists a set $E_\delta$ solution to the problem
        \[
            \J(E_\de)=\max_{V\in\M_m^\de}\J(V).
        \]
        Thanks to \Cref{prop: smalldelta} it is sufficient to prove that
        \[
            \li_{\de\to 0}\fr{\J(B_*)-\J(E_\de)}{\de^2}>0.
        \]
        Let us assume by contradiction, up to extracting a subsequence, that
        \[
            \lim_{\de\to 0}\fr{\J(B_*)-\J(E_\de)}{\de^2}=0.
        \]
        For every $\de$ we define $u_\de:=u_{E_\de}$, and $w_\de:=w_{E_\de}$ the adjoint states (see \Cref{defi: adjoint}).
        \begin{proofItemize}
            \item We first claim that there exists $t_\delta$ such that
        \[
            \widetilde{E}_\de:=\{w_\de>t_\de\}
        \]
        is of volume $m$. To that end we prove the function $t\in(0,\infty)\mapsto |\{w_\de>t\}|$ is continuous, which is the same as proving that $|\{w_\de=t\}|=0$ for every $t>0$. Let $t>0$: we argue similarly to \cite[page 7]{ChanGriesImKurOhn2000}. Because $w_\delta$ is in $H^2(B_1)$, we have $\Delta w_\de=0$ a.e. in $\{w_\de=t\}$, and therefore $j'(u_\delta)=0$ a.e. in $\{w_\de=t\}$. By assumption $j'(s)>0$ for every $s>0$ ($j'$ is strictly increasing as $j''>0$ in $(0,\infty)$), and since $\{w_\de=t\}\subset B_1$ and $u_\de>0$ in $B_1$, we obtain $|\{w_\de=t\}|=0$. This shows the existence of $t_\de$, and also that $|\{w_\de>t_\de\}|=|\{w_\de\ge t_\de\}|$. Finally, we show that $t_\de$ is unique: let $\widehat{t}$ be such that $|\{w>\widehat{t}\}|$ is also equal to $m$, and assume for example $t_\de\leq \widehat{t}$ (the other case being similar). Then the set $\{t_\de<w<\widehat{t}\}$ is open with zero measure, so it is empty. But as $m\in (0,|B_1|)$, $\min w_\de<t_\de\leq \widehat{t}<\max w_\de$, so by connectedness of $B_1$ and continuity of $w_\de$, the set $\{t_\de<w<\widehat{t}\}$ can be empty only if $t_\de=\widehat{t}$.

       \item Second, we show that $\widetilde{E}_\de$ is a $W^{2,q}$ deformation of $B_*$. First, by classical elliptic estimates (\Cref{thm: regularity}) $u_\de$ converges to $u_0$ in $W^{2,q}(B_1)$ for every $q<\infty$. As $j$ is locally $\be:=(1-\al)$-H\"older continuous (see \autoref{rem: holder}) this implies that $w_\de$ converges to $w_0$ in $W^{2,q}$ for every $q<\infty$. On the other hand, $j'(u_0)$ is $C^1$ far from $\pa B_1$, so that $w_0\in W^{3,q}(B_{1-\eps};\R^n)$ for every small $\eps$.
       Moreover, since $0<m<\abs{B_1}$, there exists a positive constant $c$ such that $\abs{\na w_0}\ge c$ on $\pa B_*$. 
       
       Therefore, we are in position to apply \Cref{lem: constmeaslevel} to $w_\de$ and \Cref{lem: convdef} to $w_\de-t_\de$, so that the following holds: for every fixed $\eps>0$ and for $\de<\de_0(\eps)$ we can find deformations $\Phi_\de\in W^{2,q}(B_1;\R^n)$ such that $\norm{\Phi_\de}_{W^{2,q}}\le \eps$, $\Phi_\de$ are orthogonal to $\pa B_*$, and
        \[
            \widetilde{E}_\de=B_*^{\Phi_\de}.
        \] 
       \item Third, by the convergence of $w_\delta$ and the regularity of $\Phi_\delta$, we find constants $c,C>0$ such that for $\delta$ small enough,
        \begin{equation}
        \label{eq: gradandperbound}
            \abs{\na w_\de}\ge c \qquad \textrm{ on } \partial B_*\quad\textrm{ and }\quad            \Hn(\{w_\de=t_\de\})=\int_{\pa B_*}\jac^{\pa B_*}(\id+\Phi_\de)\,d\Hn\le C.
        \end{equation}
       \end{proofItemize}
       We now compute, using the convexity of $j$,
\begin{equation*}
        \begin{split}
            \J(E_\de)&=\int_{B_1} j(u_\de)\,dx \le \J(\widetilde{E}_\de) - \int_{B_1}j'(u_\de)(\widetilde{u}_\de - u_\de),
        \end{split}
        \end{equation*}
       where $\widetilde{u}_\de = u_{\widetilde{E}_\de}$. We integrate by parts two times using $j'(u_\de)=-\De w_\de$, so that 
       \begin{equation*}
            \J(E_\de)\le \J(\widetilde{E}_\de) - \int_{B_1} w_\de(\chi_{\widetilde{E}_\de}-\chi_{E_\de}).
        \end{equation*}
        We now notice that thanks to the uniform estimates \eqref{eq: gradandperbound} and the quantitative bathtub principle (\Cref{eq: quantbathtub}) we get the existence of a uniform constant $c$ such that
        \begin{equation}
        \label{eq: EtotildeE}
            \J(\widetilde{E}_\de)-\J(E_\de)\ge c\abs{\widetilde{E}_\de\De E_\de}^2.
        \end{equation}
       
       On the other hand, since $\norm{\Phi}_{W^{2,q}}<\eps$, we can apply \Cref{thm: fuglede} when $\eps$ is small enough, which gives
        \begin{equation}
        \label{eq: tildeEtoBstar}
            \J(B_*)-\J(\widetilde{E}_\de)\ge c\abs{\widetilde{E}_\de\De B_*}^2,
        \end{equation}
        also for a positive constant $c$. Finally, noticing that 
        \[
            \de=\abs{E_\de\De B_*}\le \abs{E_\de\De\widetilde{E}_\de}+\abs{\widetilde{E}_\de\De B_*},
        \]
        we join \eqref{eq: EtotildeE} and \eqref{eq: tildeEtoBstar}, and we get the existence of a constant $c$ such that
        \[
            c\de^2\le 2c\ml(\abs{E_\de\De\widetilde{E}_\de}^2+\abs{\widetilde{E}_\de\De B_*}^2\mr) \\[5 pt] 
            \le \J(B_*)-\J(E_\de),
        \]
        which is a contradiction and concludes the proof.
    \end{proof}

\section{Proof of \Cref{thm: fuglede}}\label{sect:fuglede}

So far, we have proved \Cref{thm: pnorm} in the special case $p=1$, and we have reduced the proof of \Cref{thm: maintheorem} (which contains \Cref{thm: pnorm} for $p\in(1,\infty)$) to the proof of \Cref{thm: fuglede} asserting stability of $B_*$ among smooth deformations. This section is dedicated to the proof of this last result, which is based on a shape derivative approach: \red{
the idea is to consider smooth deformations $B_*^\Phi=(\id+\Phi)(B_*)$ (with same volume as $B_*$) and to use a Taylor expansion 
$$\J(B_*^\Phi)=\J(B_*)+\J'(B_*)[\Phi]+\frac{1}{2}\J''(B_*)[\Phi,\Phi]+R(\Phi)$$
where one expects $\J'(B_*)[\Phi]=0$, $\J''(B_*)[\Phi,\Phi]$ is strictly negative in a suitable sense (we refer to this as coercivity), and $R(\Phi)$ is small compare to $\J''(B_*)[\Phi,\Phi]$. This strategy is precisely described in \cite{DL19} for example. This section is therefore }divided in 4 paragraphs: 
\begin{enumerate}
\item first, the computation of first and second order shape derivatives, that leads to the expression of the Lagrangian for problem \eqref{eq:maxJ}, \red{taking into account the volume constraint,}
\item then the proof of coercivity for the second order derivative of this Lagrangian,
\item as a classical third step, we then need to prove an improved continuity property of the second order derivative\red{, allowing to control the remainder term in the Taylor expansion, which does not come ``for free'' as the regularity property in the variable $\Phi$ is obtained in a norm that is stronger than the coercivity norm from the previous step,}
\item and finally we combine all of these ingredients to conclude the proof.
\end{enumerate}

\subsection{Computation of shape derivatives}\label{ssect:shapederivatives}
In this section we compute the shape derivatives of the shape functional 
\[
E\longmapsto \J(E)=\int_{B_1}j(u_E).
\]

We start by recalling some basic facts about shape derivatives.
    \begin{defi}[Shape derivative]
    \label{defi: shapeder}
        Let $s\in(0,1)$, let $\mathcal{O}$ be a family of $C^{1,s}$ sets compactly supported in $B_1$, and let
        \[
            \mathcal{F}: \mathcal{O}\to X
        \]
        with $X$ a Banach space. Let $q>n$ large enough to have $W^{2,q}(B_1;\R^n)\hookrightarrow C^{1,s}(B_1;\R^n)$. For every set $E\in\mathcal{O}$, we say that $\F$ is shape differentiable at first or second order if the functional
        \begin{equation}\label{eq: shapeder}
        \Phi\in W^{2,q}(B_1;\R^n)\longmapsto \mathcal{F}((\id+\Phi)(E))\in 
        X. 
        \end{equation} admits Fr\'echet derivatives (of first or second order) at 0, and in this case, we 
        define the \emph{shape derivatives} (denoted $\mathcal{F}'(E)[\Phi]$ and $\mathcal{F}''(E)[\Phi,\Psi]$) as the Fr\'echet derivatives at 0 of \eqref{eq: shapeder}.
\end{defi}

We will follow the classical approach (see for instance \cite[Theorem 5.3.2]{HP18}) to prove the existence and compute the shape derivative of $u_E$, $w_E$ and finally $\J(E)$, in the next three results  respectively. In the whole section, $E$ will denote  a set of class $C^{1,s}$ such that $\overline{E}\subset B_1$, and $p$ large enough so that $W^{1,q}(B_1)\subset C^{0,s}(\overline{B_1})$. Then for every $\Phi\in W^{2,q}(B_1;\R^n)$
\[
E^\Phi:=(\id+\Phi)(E)
\]
is also of class $C^{1,s}$, and $W^{1,q}(B_1,\R^n)$ is an algebra. 

We then define the functions
\[
\begin{aligned}
&u_\Phi:=u_{E^\Phi}  &&w_\Phi:=w_{E^\Phi} \\
&\hu_\Phi:=u_{E^{\Phi}}\circ (\id+\Phi) \qquad \qquad &&\hw_\Phi:=w_{E^{\Phi}}\circ (\id+\Phi),
\end{aligned}
\]
where $w_E$ is the adjoint state defined in \Cref{defi: adjoint}.
Let us recall the classical Hadamard formula that can be found for instance in \cite[Corollary 5.2.8]{HP18}. 
\begin{lemma}[Hadamard formula]
\label{lem: distrderchar}
    Let $T>0$ and $F\in C^1([0,T);W^{1,\infty}(\R^n;\R^n))$ and we denote $F_t(x)=F(t,x)$. We assume $F_t$ invertible for every $t\in [0,T)$.  Then for every function $f\in C^1([0,T); L^1(\R^n))\cap C^0([0,T); W^{1,1}(\R^n))$
    \begin{equation}
    \label{eq: Hadamard1}
        \frac{d}{dt}\left(\int_{F_t(E)}f(t,x)dx\right)=\int_{F_t(E)}\ml(\pa_t f(t,x)+\divv(f(t,x)\,{V_t}(x))\mr)\,dx,
    \end{equation}
    where ${V_t}(x)=\pa_t F_t(F_t^{-1}(x))$.
\end{lemma}
\begin{rem}
\label{rem: firsthadamard}
    In particular, when $F_t=\id+t\Phi$ with $\Phi\in W^{1,\infty}(\R^n;\R^n)$, we define
    \[
       \tphi_t(x):=V_t(x)=\Phi\circ F_t^{-1}(x).
    \]
\end{rem}

\begin{prop}[Shape differentiability of $u_\Phi$]
\label{prop: u'}
 The application
    \[
        \Phi\in W^{2,q}(B_1;\R^n)\longmapsto u_\Phi\in W^{1,q}_0(B_1)
    \]
    is of class $C^1$ in a neighborhood of 0. In particular, if $\Phi\in W^{2,q}(B_1;\R^n)$ is small enough and we denote by $u_t:=u_{t\Phi}$, then for every $t\in[0,1]$ we have that $u'_t\in W^{1,q}_0(B_1)$ is the unique solution in $H^1_0(B_1)$ to
    \begin{equation}
    \label{eq: u'}
        -\Delta u'_t = (\tphi_t\cdot\nu_t)\,\mrestr{d\Hn}{\pa E^{t\Phi}}, 
    \end{equation}
    (in the distributional sense), where $\widetilde{\Phi}_t=\Phi\circ(\id+t\Phi)^{-1}$ and $\nu_t$ is the exterior unit normal to $E^{t\Phi}$.
\end{prop}

\begin{proof}
\begin{proofItemize}
\item  If $\norm{\Phi}_{2,q}<1$, we can invert the matrix $I_n+D\Phi$, so that, defining
\[
J_\Phi:=\det(I_n+D\Phi)\qquad \qquad A_\Phi:=J_\Phi\, (I_n+D\Phi)^{-1}(I_n+D\Phi)^{-T},
\]
we have in the distributional sense
\[
-\divv(A_\Phi \na \hu_\Phi)=\chi_E\,J_\Phi.
\]
\item We introduce 
$X=H^1_0(B_1)\cap W^{2,q}(B_1)$ and 
    \[
    \begin{array}{cccl}
    \mathcal{F}:&W^{2,q}(B_1;\R^n)\times X&\longrightarrow &L^q(B_1) \\[7 pt]
    &(\Phi,\hu) &\longmapsto &-\divv(A_\Phi \na \hu)-\chi_E\,J_\Phi.
    \end{array}\]

    Then $\F$ is of class $C^\infty$ in a neighbourhood of $0$.
    Indeed,
    we first notice that the application
    \[
        (A,\hu)\in \R^{n\times n}\times X \longmapsto -\divv(A\na \hu)\in L^p(B_1)
    \]
    is of class $C^\infty$ because it is linear and continuous in both variables. Analogously, using that $W^{1,q}$ is an algebra, multilinearity and continuity in the origin imply that the applications
    \begin{gather*}
        \Phi \in W^{2,q}(B_1) \longmapsto (I_n+D\Phi)^{-1}=\sum_{k=1}^{+\infty}(D\Phi)^k \in W^{1,q}(B_1) \\[9 pt]
        \Phi \in W^{2,q}(B_1) \longmapsto J_\Phi=\det(I_n+ D\Phi)\in W^{1,q}(B_1)
    \end{gather*}
    are of class $C^\infty$ in a neighborhood of $0$. Therefore, the application 
    \[
        \Phi \in W^{2,q}(B_1) \longmapsto A_\Phi\in W^{1,q}(B_1)
    \]
    is $C^\infty$ in a neighborhood of $0$, as well as $\F$.
\item  We are now in position to use the implicit function theorem to deduce from the previous point that the application
    \[
        \Phi\in W^{2,q}(B_1;\R^n)\longmapsto \hu_\Phi\in X,
    \]
    that satisfies $\F(\Phi,\hu_\Phi)=0$ is of class $C^\infty$ in a neighborhood of 0.
   Indeed, $\hu_0=u_E$ and the map
    \[
        \xi\in X \longmapsto D_{\hu}\F(0,u_E)[\xi]=-\De \xi - \chi_E,
    \]
    is a diffeomorphism from $X$ onto $L^p(B_1)$ thanks to the classical elliptic regularity (see for instance \cite[Theorem 9.14]{GT98}).
\item 
    Then, by noticing that $u_\Phi=\hu_\Phi\circ(\id+\Phi)^{-1}$, we know that the application
    \[
        \Phi\in W^{2,q}(B_1;\R^n)\longmapsto u_\Phi\in W^{1,q}_0(B_1)
    \]
    is $C^1$ in a neighborhood of $0$ (for instance, we can apply \cite[Lemma 5.3.3]{HP18} to $\widehat{u}_\Phi\circ (\id+\Phi)^{-1}$ and $\na\widehat{u}_\Phi\circ(\id+\Phi)^{-1}$).    
    Let us also recall that $u_{t}$ solves the weak equation:
    \begin{equation}
        \label{eq: weakELuPhi}
        \int_{B_1} \nabla u_{t}\cdot\nabla \varphi \, dx=\int_{E^{t\Phi}} \varphi \, dx \qquad \qquad \forall\varphi\in H^1_0({B_1}).
    \end{equation}
    Using \Cref{lem: distrderchar} to differentiate \eqref{eq: weakELuPhi}, we get that $u'_t$ solves
    \[
        \int_{B_1} \nabla u'_{t}\cdot\nabla \varphi \, dx=\int_{\pa {E^{t\Phi}}} \varphi\, (\tphi_t\cdot\nu_{t})\, d\Hn \qquad \qquad \forall\varphi\in H^1_0({B_1}),
    \]
    which is \eqref{eq: u'}.
\end{proofItemize}
\end{proof}

\begin{prop}[Shape differentiability of $w_\Phi$]
\label{prop: w'}
    Let $j\in C^1(\R_+)\cap C^2((0,\infty))$ such that
    \begin{equation}
    \label{eq: propw'Assumptionj''}
         \red{\lim_{s\to0^+}s^\al j''(s)\in (0,+\infty)}
    \end{equation}
    \red{for some $\al\in(-\infty,1)$}, and let $\eps\in(0,1)$. Then the application
    \[
        \Phi\in W^{2,q}(B_1;\R^n)\longmapsto w_\Phi\in 
        W^{1,q}_0(B_1)\cap W^{2,q}(B_{1-\eps})
    \]
    is of class $C^1$ in a neighborhood of 0. In particular, if $\Phi\in W^{2,q}(B_1;\R^n)$ is small enough and we denote by $w_t:=w_{t\Phi}$, then for $t\in[0,1]$ we have that $w'_t\in W^{1,q}_0(B_1)$ is the unique solution in $H^1_0(B_1)$ to
    \begin{equation}
    \label{eq: w'}
        -\Delta w'_t = j''(u_t)u'_t,
    \end{equation}
    where $u_t'$ is defined in \Cref{prop: u'}.
\end{prop}
\begin{proof}
We proceed as in the proof of \Cref{prop: u'}, with the extra difficulty that $j$ is not $C^2$ up to 0. 
\begin{proofItemize}
\item First we claim that the map
     \[
        \G: \Phi \in W^{2,q}(B_1;\R^n)\longmapsto j'(\hu_\Phi)\in L^q(B_1)\cap W^{1,q}(B_{1-\eps/2})
        \]
        is of class $C^1$ near 0. Since $j'$ is $C^1$ far from 0, we immediately have by \Cref{prop: u'} that $\Phi\mapsto j'(\hu_\Phi)\in  W^{1,q}(B_{1-\eps/2})$ is of class $C^1$. It remains to show that $\Phi\mapsto j'(\hu_\Phi)\in  L^{q}(B_{1})$ is $C^1$. As $j'$ is not assumed to be $C^1$ up to 0, as it is done for instance in \cite[Lemma 2.5]{FusZhan2017}, we approximate the functional $\G$: for small $\eps>0$ we define
        \[
            \G_\eps: \Phi \in W^{2,q}(B_1;\R^n)\longmapsto j'(\eps+\hu_\Phi)\in L^q(B_1). 
    \]
    From the proof of \Cref{prop: u'}, recalling
    \[
        X=H^1_0(B_1)\cap W^{2,q}(B_1)\subset C^1(\overline{B_1}),
    \] 
    there exists $\mathcal{U}$ neighborhood of 0 in $W^{2,q}(B_1;\R^n)$ such that the map $\Phi\in\mathcal{U}\longmapsto\hu_\Phi\in X$ is $C^1$. In particular $\norm{\hu_\Phi}_\infty$ is equibounded for $\Phi\in\mathcal{U}$, and since $j'\in C^1([\eps,+\infty))$, we get that $\G_\eps$ is $C^1$ in {$\mathcal{U}$} and that $\G_\eps(0)$ converges to $\G(0)$ in $L^q(B_1)$. Once we prove that 
    \[
\begin{array}{cccl}\G_\eps':&{\mathcal{U}}&\to&\Ll(W^{2,q}(B_1;\R^n),L^q(B_1))\\[2mm]
    &\Phi&\mapsto&\left(\eta\mapsto {j''(\eps+\widehat{u}_\Phi)\,\widehat{u}_\Phi'[\eta]}\right)\end{array}\] 
    converges uniformly in $\Phi$, from~\cite[Theorem 3.6.1]{Cartan1967} we can conclude that $\G$ is $C^1$ near 0. If we let $h(t)=tj''(t)$, then by \red{the assumption~\eqref{eq: propw'Assumptionj''} we have} $h\in C^0({\R_+})$ and by uniform continuity, $h(\eps+\hu_\Phi)$ converges to $h(\hu_\Phi)$ uniformly in $\Phi\in \mathcal{U}$. Therefore, it remains to check that 
    \begin{equation}
    \label{eq:uniformConvDeriv}
        \lim_{\eps\to 0^+}\sup_{\substack{\Phi\in\mathcal{U} \\ \norm{\eta}_{2,q}\le 1 }}\norm*{\hu'_\Phi[\eta]\ml(\frac{1}{\eps+\hu_\Phi}-\frac{1}{\hu_\Phi}\mr)}_{p} = 0.
    \end{equation}
     Up to choosing a smaller $\mathcal{U}$, there exist uniform positive constants $c,C$ such that $\norm{\na \hu'_\Phi[\eta]}_\infty\le C$, and $\abs{\na \hu_\Phi(y)}\ge c$ for every $y\in B_1\setminus B_{1-c}$ (here we used that $\abs{\na u_0}>0$ near $\pa B_1$). \red{In particular, $\abs{\hu'_\Phi[\eta](x)}\le C(1-|x|)$, and $\hu_\Phi(x)\ge c(1-|x|)$ for suitable constants $c,C>0$ and for every $x\in B_1$. }This leads to the following uniform estimate: for every $x\in B_1$, and for every $\eps\in(0,1)$,
    \[
        \abs*{\frac{\eps\,\hu'_\Phi[\eta](x)}{(\eps+\hu_\Phi(x))\,\hu_\Phi(x)}}\le \red{ \fr{C\eps(1-\abs{x})}{c(1-|x|)(\eps+c(1-\abs{x}))}\leq \fr{C}{1+c(1-\abs{x})}},
    \]
    which, by dominated convergence, ensures~\eqref{eq:uniformConvDeriv}.
    
    \item As in the proof of \Cref{prop: u'}, we obtain
    \[
    \begin{array}{cccl}
    \mathcal{F}:&W^{2,q}(B_1;\R^n)\times X&\longrightarrow &L^q(B_1) \\[7 pt]
    &(\Phi,\hu) &\longmapsto &-\divv(A_\Phi \na \hu)-j'(\widehat{u}_\Phi)\,J_\Phi.
    \end{array}\]
is $C^1$.
As \[
        D_{\hw}\F(0,w_E)[\xi]=-\De \xi - j'(u_E)
    \]
    is a diffeomorphism of $X$ onto $L^q(B_1)$, the implicit function theorem applies, and $\Phi\in W^{2,q}(B_1;\R^n)\mapsto \widehat{w}_\Phi\in X$ is of class $C^1$ in a neighborhood of 0, and 
    \[
         \Phi \in W^{2,q}(B_1;\R^n)\longmapsto w_\Phi\in W^{1,q}_0(B_1)
    \]
    is $C^1$ in a neighborhood of 0. 
    Moreover, the function $w_{t}$ solves the weak equation \eqref{eq: adjoint}, namely
    \begin{equation}
        \label{eq: weakELwPhi}
        \int_{B_1} \nabla w_{t}\cdot\nabla \varphi \, dx=\int_{B_1}  j'(u_{t})\vp \, dx \qquad \qquad \forall\varphi\in H^1_0({B_1}).
    \end{equation}
    Using \Cref{prop: u'}, we can differentiate \eqref{eq: weakELwPhi} to obtain that $w'_t$ solves
    \[
        \int_{B_1} \nabla w'_{t}\cdot\nabla \varphi \, dx=\int_{B_1} j''(u_t)u'_t\,\varphi\, dx \qquad \qquad \forall\varphi\in H^1_0({B_1}),
    \]
    which proves \eqref{eq: w'}.

\end{proofItemize}
\end{proof}

\begin{prop}[Shape derivative of $\J$]
\label{prop: firstshapeder}
Let $j\in C^1(\R_+)\cap C^2((0,+\infty))$ be such that

\[
          \red{\lim_{s\to0^+}s^\al j''(s)\in (0,+\infty)}
\]
    \red{for some $\al\in(-\infty,1)$}.  Then the application
    \[
        \Phi\in W^{2,q}(B_1;{\R^n})\longmapsto \J(E^\Phi)\in \R
    \]
    is of class $C^2$ in a neighborhood of 0.
 Moreover, if $\Phi\in W^{2,q}(B_1;\R^n)$ is small enough, then $J(t)=\J(E^{t\Phi})$ satisfies for $t\in[0,1]$,
\begin{equation} 
\label{eq: J'}
\begin{split}
    J'(t)=\int_{\partial E^{t\Phi}} w_{t}\, (\tphi_t\cdot\nu_t)\,d\Hn,
\end{split}
\end{equation}
where $w_t=w_{E^{t\Phi}}$ is the adjoint state defined in \Cref{defi: adjoint}, $\nu_t=\nu_{E^{t\Phi}}$ is the outer unit normal to $E^{t\Phi}$, and $\tphi_t=\Phi\circ(\id+t\Phi)^{-1}$. Moreover, letting $g_t=\tphi_t\cdot\nu_t$,
\begin{equation} 
\label{eq: J''}
\begin{split}
    J''(t)=\int_{\pa E^{t\Phi}}\ml(w'_tg_t+ g_t(\na w_t\cdot\tphi_t)\mr)\,d\Hn+a_t(\Phi,\Phi),
\end{split}
\end{equation}
where
\[
a_t(\Phi,\Phi)=\int_{\pa E^{t\Phi}}w_t\ml(g_t \divv(\tphi_t)-{((D\tphi_t)\tphi_t)}\cdot\nu_t\mr)d\Hn.
\]
\end{prop}
\begin{proof}
    For every $\Phi\in W^{2,q}(B_1;\R^n)$
    \[
        \J(E^\Phi)=\int_{B_1}j(u_\Phi)\,dx,
    \]
    so \Cref{prop: u'}, joint with $j\in C^1(\R_+)$, gives  that  $\Phi\mapsto\J(E^\Phi)$ is $C^1$.
    Let us denote by $u_t:=u_{t\Phi}$. We begin by noticing that
    \[
        J'(t)=\int_{B_1}j'(u_t)u'_t\,dx=\int_{B_1}\na w_t\cdot\na u'_t\,dx=\int_{\partial E^{t\Phi}}w_t\,g_t\,d\Hn
    \]
    where we used \eqref{eq: adjoint} and~\eqref{eq: u'}. Therefore, we have the following shape derivative
    \[
        \J'(E_\Phi)[\Psi] = \int_{\pa E^\Phi} w_{E^\Phi}(\Psi\cdot\nu_{E^\Phi})\,d\Hn=\int_{E^\Phi} \divv(w_{E^\Phi}\Psi)\,dx.
    \]
   By \Cref{prop: w'} we know that the map $\Phi\mapsto w_{E^\Phi}\in W^{2,q}(B_{1-\eps})$ is of class $C^1$ in a neighbourhood of 0. In particular, if $\Phi$ and $\eps$ are small, we have $E^\Phi\sbs B_{1-\eps}$. Finally, since for every fixed $f\in W^{1,1}$ the map $\Phi\mapsto \int_{E^\Phi}f\,dx$ is of class $C^1$ in a neighbourhood of 0 (see for instance~\cite[Theorem 5.2.2]{HP18}), then $\J$ is of class $C^2$.
    We now compute again $J'(t)$ as
    \[
        J'(t)=\int_{E^{t\Phi}}\divv(w_t \tphi_t)\,dx,
    \]
    so that by the Hadamard's formula \Cref{lem: distrderchar} we get
    \[
        J''(t)=\int_{\pa E^{t\Phi}}\divv(w_t \tphi_t)\tphi_t\cdot\nu_t\,d\Hn+\int_{\pa E^{t\Phi}}(w'_t\tphi_t+w_t \, \pa_t\tphi_t)\cdot\nu_t\,d\Hn.
    \]
    Formula \eqref{eq: J''} finally follows by differentiating in $t$ the equation $\tphi_t\circ (\id+t\Phi)=\Phi$, leading to
    \[
        \pa_t \tphi_t = - (D\tphi_t)\tphi_t.
    \]
\end{proof}
\begin{rem}
    \label{rem: j'in0}
    When we evaluate \eqref{eq: J'} in $t=0$, we get 
    \[
        J'(0)=\int_{\pa E}w_E(\Phi\cdot\nu_E)\,d\Hn.
    \]
    When $\Phi$ is orthogonal to $\pa E$ (i.e. $\Phi=g_0\nu_0$ on $\pa E$), then \eqref{eq: J''} in $t=0$ reads 
    \[
        J''(0)=\int_{\pa E}\mleft(w_0'\,g_0+w_0 H_E g_0^2+\fr{\pa w_0}{\pa \nu_0}g_0^2\mright)\,d\Hn,
    \]
    with $H_E$ the mean curvature of $\pa E$. Indeed, notice that in this case
    \[
        g_0\divv(\Phi)-((D\Phi)\Phi)\cdot\nu_0=g_0^2\divv^{\pa E}(\nu_0)=g_0^2 H_E.
    \]
\end{rem}

\medskip
The previous computations allows us to write the Lagrangian for the maximizing problem
\begin{equation}
    \label{eq: maxproblem}
        \max_{\abs{E}=m}\J(E).
    \end{equation}
    In the following, for $x$ close to $\pa B_*$, we let $\pi_{\pa B_*}(x)$ be the unique projection of $x$ onto $\pa B_*$, and we let $\nu_0(x)=\nu_0(\pi_{\pa B_*}(x))$ be the extended unit normal to $\pa B_*$.
\begin{cor}\label{cor: lagrangian}
\red{Under the same assumptions of~\Cref{prop: firstshapeder}, }for $\tau\in\R$ and $E\subset B_1$, we define 
 \[
        \Ll_\tau(E):=\J(E)+\tau\abs{E}.
    \]
For $m\in(0,|B_1|)$ and $E=B_*$ the centered ball of volume $m$, we set
    \begin{equation} 
    \label{eq: tau}
        \tau=-\restr{w_0}{\partial B_*}.
    \end{equation}
    Then $\Phi\in W^{2,q}(B_1;\R^n)\mapsto \Ll_\tau(B_*^\Phi)$ is of class $C^2$ near 0, and
    \begin{enumerate}[label=(\roman*)]
    \item $\Ll_\tau'(B_*)\equiv 0$
    \item for every $\Phi\in W^{2,q}(B_1;\R^n)$ small enough such that $\Phi$ is normal on $\partial B_*$ and $\widetilde{\Phi}_t=\Phi$ for $t\in[0,1]$, if we denote $L(t)=\Ll_\tau(B_*^{t\Phi})$ then for every $t$,
    \begin{equation}
	\label{eq: L''t}
        L''(t)=\int_{\pa B_*^{t\Phi}}\ml(w'_t g+g^2\fr{\pa w_t}{\pa \nu_0}\mr)(\nu_t\cdot\nu_0)\,d\Hn+\int_{\pa B_*^{t\Phi}}(w_t+\tau)(\nu_t\cdot\nu_0)g^2\divv(\nu_0)\,d\Hn,   
\end{equation}
    where $g=\Phi\cdot\nu_0$.
    \end{enumerate}
    \end{cor}
    \begin{proof}
        Let $V(t)=\abs{B_*^{t\Phi}}$. From Lemma \ref{lem: distrderchar} we have
        \[
                V'(0)=\int_{\partial B_*}(\Phi\cdot\nu)\,d\Hn,
        \]
        so \textit{(i)} follows from \eqref{eq: J'}.

        Using two times \autoref{lem: distrderchar} gives (see also~\cite[Proof of Proposition 4.1]{DL19}),
        \[
            V''(t)=\int_{\pa B_*^{t\Phi}}\divv(\Phi)(\Phi\cdot\nu_t)d\Hn
        \]
so with \autoref{prop: firstshapeder} and the fact that $0=\pa_t \tphi_t$, we get, letting $g_t=\Phi\cdot \nu_t$,
\[L''(t)=\int_{\pa B_*^{t\Phi}}w_t'g_t+(\na w_t\cdot\Phi)g_t+\int_{\pa B_*^{t\Phi}}(w_t+\tau)g_t\divv(\Phi).\]
        Now, as $\Phi$ is normal on $\pa B_*$ and $\Phi\circ(\id+t\Phi)=\Phi$, we have for every $t$ that $\Phi_{|\pa B_*^{t\Phi}}=g\nu_0$,  and we conclude using $0=\pa_t\tPhi_t=-(D\Phi)\Phi=-g\na g\cdot\nu_0$ so that $g\divv(\Phi)=g^2\divv(\nu_0)$.
        \end{proof}

    \begin{rem}
        The general scheme of proof of \Cref{thm: maintheorem} was strongly inspired by \cite{M21}; nevertheless, it seems to us that \cite[Section 2.5.7 and formula before (73)]{M21} are incomplete (note that these computations are also used in~\cite{MRB22,CMP23}). Compared to our formula \eqref{eq: L''t}, the first discrepancy is due to the assumption $\tphi_t=\Phi$. This seems to be missing in \cite{M21}, but without such assumption we should have
        \[
            L''(t)=\Ll_\tau''(B_*^{t\Phi})[\tphi_t,\tphi_t]+\Ll_\tau'(B_*^{t\Phi})[\pa_t \tphi_t]
        \]
       while in \cite{M21} it is used $L''(t)=\Ll''_\tau(B_*^{t\Phi})[\Phi,\Phi]$, which seems incorrect to us without assuming $\tphi_t$ constant in $t$. The second discrepancy is about the outer unit normal appearing in the equality: $\nu$ in~\cite[Formula before (73)]{M21} represents the outer unit normal to $\pa B_*$, but it should be replaced with $\hnu_t=\nu_t\circ(\id+t\Phi)$.
    \end{rem}

\begin{rem}
    We also point out a computation mistake in \cite[Proof of Proposition 4.1]{DL19} where $H$ should be replaced with $\widehat{\divv\nu_0}=(\divv\nu_0)\circ(\id+t\Phi)$. Indeed, $\divv \nu_0(x)$ is not the mean curvature of $\pa E$ in $\pi_{\pa E}(x)$, it represents, up to the sign, the mean curvature in $x$ of the boundary of the outer parallel set 
    \[
        \pa \Set{y\in \R^n | d(y,\pa E)\le d(x,\pa E)}.
    \]
\end{rem}
\subsection{Step 2: coercivity of the second order shape derivative in $B_*$}\label{ssect:coercivity}

The goal of this section is to prove a sufficient optimality condition for problem \eqref{eq: maxproblem}:

\begin{prop}
    \label{prop: eigentostab}
    Let $m\in(0,\abs{B_1})$, $B_*$ the centered ball of volume $m$, and $j\in C^1(\R_+)\cap C^2((0,+\infty))$ such that $j'(0)\geq 0$, $j''>0$ on $(0,\infty)$ and
    \[
           \red{\lim_{s\to0^+}s^\al j''(s)\in (0,+\infty)}
    \]
    \red{for some $\al\in(-\infty,1)$}. 
    Then there exist positive constants $c=c(j,m), \eta=\eta(j,m)$ such that for every $\Phi\in W^{2,q}(B_1;\R^n)$ such that $|B_*^\Phi|=|B_*|$ and $\norm{\Phi}_{\infty}<\eta$ we have
    \begin{equation*}
       \Ll_\tau''(B_*)[\Phi,\Phi]\le - c\norm{\Phi\cdot \nu_0}_{L^2(\pa B_*)}^2,
    \end{equation*}
    where $\Ll_\tau$ is the Lagrangian defined in \Cref{cor: lagrangian}.
\end{prop}

The rest of this section is devoted to the proof of \Cref{prop: eigentostab}, see the two substeps below. To simplify the computation of $\Ll_\tau''(B_*)$, we first remark that this quadratic form depends only on the trace of $\Phi$ on $\pa B_*$ (this is a well known fact, see for instance~\cite[proof of Theorem 2.3]{LNP16}). We can therefore replace $\Phi$ by $\Psi(x):=\Phi(\pi_{\pa B_*}(x))$ in a neighborhood of $\partial B_*$ and extend it smoothly to $B_1$. It is also well-known (see \cite[Theorem 5.9.2]{HP18}) that as $\Ll_\tau'(B_*)=0$, $\Ll_\tau''(B_*)$ only depends on $g=(\Phi\cdot\nu_0)_{|\partial B_*}$, so we can also assume $\Phi$ to be normal on $\pa B_*$. With these adaptations, we have $\widetilde{\Phi}_t=\Phi$ for every $t$, and we can apply \autoref{cor: lagrangian}, and $\Ll_\tau''(B_*)$ will reduce to a quadratic form on $L^2(\partial B_*)$ (see also~\cite{CMP23} and~\cite[Lemma 35]{CMP23supp}).

\subsubsection*{Substep 1: Rewriting $\Ll_\tau''(B_*)$}

We will compute $\Ll_\tau''(B_*)$ in terms of $g=(\Phi\cdot\nu_0)_{\pa B_*}$, defining a suitable eigenvalue problem that has been introduced in \cite{CMP23} (see in particular \cite[Theorem III]{CMP23} and \cite[Proposition 34]{CMP23supp}): 
from \Cref{prop: u'} and \Cref{prop: w'} we can consider $u'_0:=u'_{B_*}[\Phi]$ and $w'_0:=w'_{B_*}[\Phi]$ that are the unique solution in $H^1_0(B_1)$ to the equations
\[
-\De u'_0=g\, \mrestr{d\Hn}{\pa B_*} \qquad \qquad -\De w'_0=j''(u_0)u'_0.
\]
Using \Cref{cor: lagrangian} we get
\[
{\Ll_\tau''(B_*)[\Phi,\Phi]}=\int_{\pa B_*}\mleft(g w_0'-\abs*{\fr{\pa w_0}{\pa {\nu_0}}}g^2\mright)\,d\Hn.
\]
For every $g\in L^2(\pa B_*)$ we consider $(U_g,W_g)\in H_0^1(B_1;\R^2)$ the unique solution to the coupled boundary value problems
\[
-\De U_g= g\mrestr{d\Hn}{\pa B_*} \qquad \qquad -\De W_g=j''(u_0)U_g.
\]
Denoting $\tr_{\pa B_*}:W^{1,2}(B_1)\to L^2(\pa B_*)$ the trace operator on $\partial B_*$, we also define the operator
\[
    T: g\in L^2(\pa B_*)\longmapsto \tr_{\pa B_*}(W_g)\in L^2(\pa B_*),
\]
which is symmetric, as $\int_{\partial B_*}hW_gd\Hn=\int_{B_1}j''(u_0)U_gU_h$.
Then for $h,g\in L^2(\pa B_*)$ we define
\[
l_2(h,g):=\int_{\pa B_*} h\mleft(W_g-\abs{\pa_\nu w_0}g\mright)\,d\Hn
=\int_{\pa B_*} hT g\,d\Hn-\abs{\pa_\nu w_0}_{\pa B_*}\int_{\partial B_*}hg\,d\Hn\]
Note that, when $g=g_\Phi=(\Phi\cdot\nu_0)_
{|\partial B_*}$, we have $U_g= u'_0$ and $W_g= w'_0$, and therefore
\[
    \Ll_\tau''(B_*)[\Phi,\Phi]=l_2(g,g).
\]

\subsubsection*{Substep 2: diagonalization of $T$}

In the following, when $n\geq 2$ (the case $n=1$ will be dealt with separately, see below in the proof of \Cref{prop: eigentostab}) and $k\in\N$ we denote $(Y_{k,m})_{1\leq m\leq M(k)}$ the real spherical harmonics of degree $k$ (that is to say, a basis of the space of homogeneous, real and harmonic polynomials of degree $k$), and it is well known that these are eigenfunctions of the opposite of the spherical Laplacian $-\Delta_{\Spn}$, with eigenvalue $\La_k=k(k+n-2)$ ($M(k)$ being its multiplicity).

In this step we show that $T$ is diagonalizable and that the spherical harmonics are its eigenfunctions:
\begin{prop}
    \label{prop: computeigenv}
    Let $n\ge 2$, and let $j\in C^1(\R_+)\cap C^2((0,\infty))$ be such that $j'(0)\geq 0$, $j''>0$ and 
    \[
          \red{\lim_{s\to0^+}s^\al j''(s)\in (0,+\infty)}
    \]
    \red{for some $\al\in(-\infty,1)$}. 
    Then for every $k\in\N^*$, the spherical harmonics of degree $k$, $(Y_{k,m})_{1\leq m\leq M(k)}$ are eigenfunctions of $T$ with the same corresponding eigenvalue  $1/\lambda_k$. Moreover, the sequence $(\la_k)_{k\in\N^*}$ is non-decreasing in $k$ and we have
    \begin{equation}\label{eq:la>rho}\lambda_1>\left(|\partial_{\nu_0} w_0|_{|\pa B_*}\right)^{-1},
    \end{equation}
    where $w_0=w_{B_*}$ is the adjoint state of $u_0=u_{B_*}$.
\end{prop}
\begin{rem}
    \red{We observe that the constant function $Y_0$ is an eigenfunction for the operator $T$, but in general $\la_0$ does not satisfy~\eqref{eq:la>rho}. This highlights that for the coercivity it is necessary to take into account the volume constraint and to} work in the space $\Set{g\in L^2(\pa B_*)|\int_{\pa B_*}gd\Hn=0}$. \red{We point out that in~\cite[Lemma 36]{CMP23supp} the authors prove~\eqref{eq:la>rho} for $\la_0$ under the big mass regime $m\sim \abs{B_1}$.}
\end{rem}
\begin{proof}
    \begin{proofItemize}
        \item
              Let $k\geq 1$ be fixed. Let us define $\abs{x}=r$ and $\abs{x}^{-1}x=\theta$; in the following computations we will write $u_0(r)$ for $u_0(x)$ with $|x|=r$.  To show that $Y_{k,m}(\theta)$ are eigenfunctions for $T$, we look for a nontrivial function
              \[
                  W(x)=\vp(r)Y_{k,m}(\tht),
              \]
              solving
              \begin{equation}\label{eq:W}\left\{\begin{array}{ccl}(-\Delta)\left(\frac{1}{j''(u_0)}(-\Delta)\right)W=\lambda W\mrestr{d\Hn}{\pa B_*}&\textrm{ in }&B_1,\\[2mm]W=\left(\frac{1}{j''(u_0)}\right)\Delta W=0&\textrm{ on }&\pa B_1,\end{array}\right.
              \end{equation}
              for some $\lambda\in\R$.

              Let us denote by $\De_r$ and $\D_r^k$ the differential operators on functions of one variable, defined as
              \begin{equation}\label{eq:Dr}
                  \De_r \vp=r^{1-n}\pa_r (r^{n-1}\pa_r\vp), \qquad\quad\D_r^k\vp=\De_r\vp -\fr{\La_k}{r^2}\vp=r^{1-n-k}\pa_r\ml(r^{n-1+2k}\pa_r\ml(r^{-k}\vp\mr)\mr),
              \end{equation}
              so that
              \[
                  \De(\vp(r)Y_{k,m}(\tht))=\ml(\De_r +
                  \fr{1}{r^2}\De_{\Spn}\mr)(\vp(r)Y_{k,m}(\tht))=Y_{k,m}(\tht)\D_r^k\vp(r).
              \]
              As a consequence, we look for $\varphi$ solution of the ODE:
              \[\left\{\begin{array}{l}\D_r^k \ml(\fr{1}{j''(u_0)}\D_r^k\vp\mr)(r)=\lambda\varphi\delta_{r_*}\textrm{ in }(0,1),\\[2mm]\varphi(1)=\left(\frac{1}{j''(u_0)}\right)\D_r^k\varphi(1)=0\end{array}\right.
              \]
              and such that the behavior of $\varphi$ at 0 makes $W$ smooth enough near the origin.

              Using the second formulation of $\D_r^k$ in \eqref{eq:Dr}, we introduce four independent functions
              \[
                  \psi_{1}(r)=r^k, \qquad \qquad \psi_{2}(r)=r^{2-n-k}, \]
              \[\textrm{ and }\qquad\widetilde\psi_{\beta}(r)=r^k\int_{r_*}^r s^{-(n-1+2k)}\int_{r_*}^s j''(u_0(t))\, \psi_\beta(t)t^{n+k-1}\,dt \,ds,\qquad \qquad \beta=1,2,
              \]
              which respectively solve the equations
              \begin{equation}
                  \label{eq: psibeeq}
                  \D_r^k\psi_{\be} = 0, \qquad\qquad
                  \D_r^k\widetilde\psi_{\be} = j''(u_0)\psi_{\be}, \qquad \be=1,2.
              \end{equation}
              We now define $\vp$ in the following way
              \[
                  \vp(r)=\begin{dcases}
                      c_1^-\psi_1+c_2^-\psi_2 +c_3^-\widetilde\psi_1+c_4^-\widetilde\psi_2 & r\le r_*, \\
                      c_1^+\psi_1+c_2^+\psi_2 +c_3^+\widetilde\psi_1+c_4^+\widetilde\psi_2 & r>r_*.
                  \end{dcases}
              \]
              \begin{table*}[]
                \renewcommand{\arraystretch}{1.3} %
                \centering
                \begin{tabular}{@{}CCCCcCCCcC@{}}\toprule
                              &\multicolumn{3}{C}{\psi} &\phantom{abc} &\multicolumn{3}{C}{\fr{1}{j''(u_0)}\D_r^k\psi} &\phantom{abc} & \multicolumn{1}{C}{\pa_r(r^{-k}\psi)} \\ 
                              \cmidrule{2-4} \cmidrule{6-8} \cmidrule{10-10}
                              &r=0 &r=r_* &r=1 && r=0 &r=r_* &r=1 && \multicolumn{1}{C}{r=r_*} \\
                              \midrule
                              \psi_1 &0&r_*^k&1&&0&0&0&& 0\\
                              \psi_2 &+\infty& & &&0&0& && (2-n-2k)r_*^{1-n-2k}\\
                              \widetilde\psi_1 &0&0&\widetilde\psi_1(1)&&0&r_*^k&1&& 0\\
                              \widetilde\psi_2 & &0&\widetilde\psi_2(1)&&+\infty&r_*^{2-n-k}&1&& 0\\
                              \bottomrule
                            \end{tabular}
                \caption{Evaluation table for $\psi_\be$ and $\widetilde{\psi}_\be$. Empty cells are not needed for the computation.}
                \label{tab: evaluation}
              \end{table*}
              The behavior of $\varphi$ near $0,r_*$ and $1$ will add 7 independent conditions (see \Cref{tab: evaluation} for the values of $\psi_\be$ and its derivatives):
              \begin{itemize}
                  \item near 0, we require $\varphi(0)=0$: indeed, if this is not the case, as $Y_{k,m}$ is not constant, then $\varphi(r)Y_{k,m}(\theta)$ cannot be continuous at 0. Similarly, we require $\D_r^k\varphi(0)=0$ so that $\Delta W$ is continuous at 0.

                        As $\D_r^k\psi_1(0)=\D_r^k\psi_2(0)=\D_r^k\widetilde\psi_1(0)=0$ while $\D_r^k\psi_2(0)=+\infty$, the first condition is $c_4^-=0$.

                        And as $\psi_1(0)=\widetilde{\psi}_1(0)=0$ while $\psi_2(0)=+\infty$, the second condition reads $c_2^-=0$.
                  \item near $r_*$, we require three continuity conditions, namely for $\varphi, \:\pa_r(r^{-k}\varphi)$ and $\frac{\D^k_r\varphi}{j''(u_0)}$: as we have chosen $\widetilde{\psi}_\beta$ so that they vanish at $r_*$, these respectively lead to the three equations:
                        \[
                         \begin{cases}
	       c_1^-=c_1^+                                             \\[2mm]
                                0=c_2^-=c_2^+\\[2mm]
                                c_3^-r_*^k=c_3^+r_*^k+c_4^+r_*^{2-n-k}
\end{cases}
                        \]
                  \item finally, the two boundary conditions $\varphi(1)=\frac{1}{j''(u_0)}\D_r^k\varphi(1)=0$ write:
                        \[
                             \begin{cases}
	   c_1^+ + c_3^+\widetilde{\psi}_1(1)+c_4^+\widetilde{\psi}_2(1)=0 \\[2mm]
                                c_3^++c_4^+=0
                                
\end{cases}
                        \]
                        where \red{$\widetilde\psi_1(1)$ and $\widetilde\psi_{2}(1)$ are well defined and finite because $j''(u_0)\le C u_0^{-\al}$, which is integrable because $u_0$ decays linearly to 0 near $\pa B_1$.}
              \end{itemize}
              We therefore have 7 linear conditions for 8 variables, so there is at least a one dimensional vectorial space of solutions, and we choose $\varphi$ a nontrivial one. Notice also that the 7 conditions are independent: if $c_1^-=0$, then all the constants $c_i^{\pm}$ need to be zero, as by direct computation ($t^{k}<t^{2-n-k}$ for $t\in(0,1)$)     
              \[
                \widetilde\psi_{1}(1)-\widetilde\psi_2(1)<0.
              \]
              To conclude that $W=\varphi(r)Y_{k,m}(\theta)$ is a solution to \eqref{eq:W}, it remains to observe that $\varphi(r_*)\neq 0$, so that
              \[\lambda:=\frac{\left[\pa_r\left(\frac{\D_r^k\varphi}{j''(u_0)}\right)\right](r_*)}{\varphi(r_*)}\]
              is well defined, where
              $[v](r_*)$ denotes the jump of the function $v$ at $r_*$; this is the case as $\varphi(r_*)=c_1^-r_*^k$, and we have just observed that if $c_1^-=0$, then $\varphi\equiv0$.

        \item We now prove the monotonicity of $(\lambda_k)_k$:\\
              For $k\geq 1$, let $\varphi_k$ one of the solutions computed in the previous item. We choose the scaling so that  \begin{equation}\label{eq:scaling}\lambda_k\varphi_k(r_*)=1.
                  \end{equation}
              
              We make the proof in two steps: first we show that $\vp_k\ge 0$, and then we study the difference $\psi_k:=\vp_{k+1}-\vp_k$:

              \textbf{Step 1: }let
              \[
                  \Phi_k=\fr{1}{j''(u_0)}\D_r^k\vp_k,
              \]
              where $\varphi_k$ is one solution computed in the previous item, and
              so that
              \begin{equation}
                  \label{eq: ODEPhik}
                  \begin{dcases}
                      \De_r \Phi_k=\fr{\La_k}{r^2}\Phi_k+\delta_{r_*} & \textrm{ in }(0,1), \\[5 pt]
                      {\Phi_k(0)=}\Phi_k(1)=0,
                  \end{dcases}
              \end{equation}
              where $\de_{r_*}$ denotes the Dirac mass concentrated in $r_*$. We claim that $\Phi_k\le 0$. Indeed, let us assume that $\bar{r}$ is a maximum point for $\Phi_k$. If $\bar{r}=1$ {or $\bar{r}=0$}, the claim is proved. On the other hand, the Dirac mass in \eqref{eq: ODEPhik} implies that
                  \[
                      \Phi_k'(r_*^+)-\Phi_k'(r_*^-)=\lambda_k\varphi_k(r_*)>0,
                  \]
              which ensures that $r_*$ cannot be a maximum point. Therefore, we may evaluate the ODE \eqref{eq: ODEPhik} in the maximum point $\bar{r}$ obtaining, after the expansion of $\De_r$,
              \[
                  0\ge \Phi''_k(\bar{r})=\De_r \Phi_k(\bar{r})=\fr{\La_k}{\bar{r}^2}\Phi_k(\bar{r}),
              \]
              which proves the claim.

              Since $\Phi_k\le 0$, then
              \[
                  \begin{dcases}
                      \De_r \vp_k\le \fr{\La_k}{r^2}\vp_k & \qquad r\in(0,1), \\[5 pt]
                      {\vp_k(0)=}\vp_k(1)=0.
                  \end{dcases}
              \]
              With the same argument as above for $\Phi_k$, simplified by the fact that $\vp_k$ is of class $C^2$ in the whole interval $(0,1)$, we get (by looking at its minimum) that  $\vp_k\ge 0$.

              \noindent\textbf{Step 2:} let $\psi_k:=\vp_{k+1}-\vp_k$, and let $\Psi_k:=\Phi_{k+1}-\Phi_k$, and note that both are $C^2(0,1)$. We claim that $\Psi_k\ge 0$. First we notice that, since $\Phi_k\le 0$ and $\Lambda_k<\Lambda_{k+1}$,
              \[
                  \ml(\De_r-\fr{\La_{k+1}}{r^2}\mr)\Phi_{k+1}=\delta_{r_*}=\ml(\De_r-\fr{\La_k}{r^2}\mr)\Phi_k\le \ml(\De_r-\fr{\La_{k+1}}{r^2}\mr)\Phi_k,
              \]
              so that
              \begin{equation}
                  \label{eq: Psik}
                  \ml(\De_r-\fr{\La_{k+1}}{r^2}\mr)(\Psi_{k})\le 0.
              \end{equation}
              As done for $\vp_k$ in Step 1, \eqref{eq: Psik} ensures that $\Psi_k\ge 0$.

              Finally, $\Psi_k\ge0$ reads
              \[
                  0\le \ml(\De_r-\fr{\La_{k+1}}{r^2}\mr)\vp_{k+1}-\ml(\De_r-\fr{\La_{k}}{r^2}\mr)\vp_k\le\ml(\De_r-\fr{\La_{k+1}}{r^2}\mr)(\vp_{k+1}-\vp_k),
              \]
              where we have used in the last inequality $\La_k\vp_k\le \La_{k+1}\vp_k$. As before, this implies $\vp_{k+1}\le \vp_k$, which because of \eqref{eq:scaling} leads to $\lambda_k\le\lambda_{k+1}$.
        \item
              We can now compute the eigenvalue $\la_1$: we drop the exponent $k$ in the notations as its value will remain 1 in this paragraph. Using equations \eqref{eq: psibeeq} and $c_2^-=c_4^-=0$, we obtain
              \[
                  \la_1 =\fr{\ml[\pa_r \ml(\frac{\D_r\vp_1}{j''(u_0(r))}\mr)\mr](r_*)}{\vp_1(r_*)} \\[7 pt]
                  =\fr{c_3^+-c_3^-+\frac{1-n}{r_*^n}c_{4}^+}{c_1^-r_*}.
              \]
              Let us recall the previous system when $k=1$:
              \[
              \begin{cases}
                    c_2^-=c_2^+=c_4^-=0 \\[2mm]
                      c_1^-=c_1^+                \\[2mm]
                      c_3^-r_*=c_3^+r_*+c_4^+r_*^{1-n}                                    \\[2mm]
                        c_1^+ +c_3^+ \widetilde{\psi}_1(1)+ c_4^+\widetilde{\psi}_2(1)=0 \\[2mm]
                      c_3^++c_4^+=0
                      
\end{cases}
              \]
              From the third equation, we compute $c_3^+-c_3^-$ in terms of $c_4^+$ and therefore
              \[\lambda_1=-\frac{n}{r_*^{n+1}}\frac{c_4^+}{c_1^-}.\]
              The second equation joint with the fourth one and the fifth one expresses $c_1^-$ in terms of $c_4^+$, which eventually leads to
              \begin{equation}
                  \label{eq: rhola1}
                  \la_1=\frac{n}{r_*^{n+1}\Big[\widetilde{\psi}_2(1)-\widetilde{\psi}_1(1)\Big]}.
              \end{equation}
              Now observe that by Fubini's theorem
              \begin{equation}
                  \label{eq: psibe1}
                  \begin{split}
                    \widetilde{\psi}_2(1)-\widetilde{\psi}_1(1) & =\int_{r_*}^1 s^{-n-1}\int_{r_*}^s j''(u_0(t))\,(t^{1-n}-t)t^{n}\,dt\,ds
                       \\[7 pt]
                       &=\int_{r_*}^1 j''(u_0(t))\,t(1-t^n)\int_{t}^1s^{-n-1}\,ds\,dt
                        \\[7 pt]
                                              & =\fr{1}{n}\int_{r_*}^1j''(u_0(t)) \, t^{1-n}(1-t^n)^2\,dt                 
                  \end{split}
              \end{equation}
              This leads to
              \begin{equation}
                  \label{eq: la>1}
                  \fr{1}{\la_1}=\fr{1}{n^2 r_*^{n-1}}\int_{r_*}^1 j''(u_0(t))\, r_*^{2n}\,(1-t^n)^2\,t^{1-n}\,dt
              \end{equation}

On the other hand, we compute $|\partial_\nu w_0|_{|\partial B_*}$: we first notice
\[
\begin{split}
    \left|\frac{\pa w_0}{\pa \nu}\right|_{|\pa B_*}
    &=\fr{1}{P(B_*)}\int_{\pa B_*}\ml(-\frac{\pa w_0}{\pa \nu}\mr)\,d\Hn = \fr{1}{P(B_*)}\int_{B_*}\ml(-\De w_0\mr)\,dx \\
    &=\fr{1}{P(B_*)}\int_{B_1}(-\De u_0)j'(u_0)\,dx \ge\fr{1}{P(B_*)}\int_{B_1}j''(u_0)\abs{\na u_0}^2\,dx,
\end{split}
\]
as $j'(u_0)\in W^{1,1}(B_1)$ and $j'\ge 0$.
Then, identifying $u_0(x)=u_0(\abs{x})$, we observe
that $\na u_0(x)=\pa_r u_0(\abs{x})$ and that $u_0$ solves
\[
\begin{dcases}
    -\pa_r (r^{n-1}\pa_r u_0(r))=r^{n-1}\chi_{(0,r_*)}(r) \qquad r\in(0,1), \\[7pt]
    \pa_r u_0(0)=0,\\[7pt]
    u_0(1)=0.
\end{dcases}
\]
Hence,
\[
\pa_r u_0(r)=-r^{1-n}\int_0^{r\wedge r_*}s^{n-1}\,ds=-\fr{(r\wedge r_*)^{n}}{nr^{n-1}}\,,
\]
which implies 
\begin{equation}
\label{eq: explicitrho}
    |\pa_\nu w_0|_{|\pa B_*}\ge \fr{1}{n^2 r_*^{n-1}}\int_0^1 j''(u_0(t))(t\wedge r_*)^{2n}\, t^{1-n}\,dt.
\end{equation}
By comparing the integrands in \eqref{eq: la>1} and \eqref{eq: explicitrho}, we easily conclude $\la_1>(|\partial_\nu w_0|_{|\pa B_*})^{-1}$.
\end{proofItemize}
\end{proof}

\subsubsection*{Conclusion of the proof of \Cref{prop: eigentostab}}
\begin{proof}[Proof of \Cref{prop: eigentostab}]
\begin{proofItemize}
    \item    In the case $n=1$, $L^2(\pa B_*)$ is a two dimensional space, $B_*^\Phi$ is an interval, and because we assume $|B_*^\Phi|=|B_*|$, we can restrict to $\Phi$ being a translation, in which case $g(x)=\sgn(x)\al$ for some $\al\in(r_*-1,1-r_*)$. Let us denote by $W=W_{\sgn}$. We notice that the functions $W$ and $\pa_\nu w_0$ are odd, so that the shape derivative computed in \Cref{cor: lagrangian} reads 
        \[            \Ll_\tau''(B_*)[\Phi,\Phi]=l_2(g,g)=2\al^2\ml(W(r_*)-\abs{\pa_\nu w_0}(r_*)\mr),
        \]
        Finally, reproducing the proof of \eqref{eq:la>rho} in \Cref{prop: computeigenv} we get that $W(r_*)<\abs{\pa_\nu w_0}(r_*)$, and this gives 
        \[\Ll_\tau''(B_*)[\Phi,\Phi]\leq -c\|\Phi\cdot\nu_0\|^2_{L^2(\pa B_*)}.\]
\item Assume now $n\geq 2$.    Let $\Phi$ be as in the proposition. We denote $g=\Phi\cdot{\nu_0}_{|\pa B_*}$, and recall that
\[
\Ll_\tau''(B_*)[\Phi,\Phi]=l_2(g,g)=\int_{\pa B_*}gTg-|\partial_\nu w_0|_{|B_*}\int_{\pa B_*}g^2.
\]
From the volume constraint, if $\|g\|_\infty\leq \eta$ then
\[
    0=\fr{1}{n}\int_{\pa B_*}\ml((1+g)^n-1\mr)\,d\Hn=\al_0+\fr{1}{n}\sum_{k=2}^n \binom{n}{k}\int_{\pa B_*}g^k\,d\Hn\ge \abs{\al_0}-C\eta\norm{g}_2,
\]
where $\alpha_0=\int_{\pa B_*}g$. 
We define $\widetilde{g}=g-\al_0$ and we first study $l_2(\widetilde{g},\widetilde{g})$:
as $(Y_{k,m})_{k\in\N^*, 1\le m\le M(k)}$ is an orthogonal basis of $\set{h\in L^2(\pa B_*) | \int_{\pa B_*}hd\Hn=0}$, we can  
 decompose $\widetilde{g}=\sum_{k\geq 1,m} \al_{k,m} Y_{k,m}$ and we get from \Cref{prop: computeigenv}:
\[\begin{split}
        l_2(\widetilde{g},\widetilde{g})&=\sum_{k\geq 1,m}\frac{1}{\la_k}\alpha_{k,m}^2\|Y_{k,m}\|^2_{L^2(\pa B_*)}-|\pa_\nu w_0|_{|\pa B_*}\sum_{k\geq 1,m}\alpha_{k,m}^2\|Y_{k,m}\|^2_{L^2(\pa B_*)}\\&\le \mleft(\frac{1}{\la_1}-|\pa_\nu w_0|_{|\pa B_*}\mright)\norm{\widetilde{g}}_{L^2(\pa B_*)}^2,
        \end{split}
\]
so from \eqref{eq:la>rho} there exists $c>0$ such that
\[l_2(\widetilde{g},\widetilde{g})\leq -c\|\widetilde{g}\|^2_{L^2(\pa B_*)}.\]
Now, as $l_2$ is continuous on $L^2(\pa B_*)$, we get
\begin{eqnarray*}\Ll_\tau''(B_*)[\Phi,\Phi]=l_2(g,g)&=&l_2(\widetilde{g},\widetilde{g})+2l_2(g,\al_0)-l_2(\al_0,\al_0)\\&\leq& -c\|g-\alpha_0\|_{L^2(\pa B_*)}^2+2|\al_0|\|g\|_{L^2(\pa B_*)}+|\al_0|^2
\end{eqnarray*}
and we conclude using $\|g-\al_0\|_{L^2(\pa B_*)}\geq \|g\|_{L^2(\pa B_*)}-|\al_0|$ and $|\al_0|\leq C\eta \|g\|_{L^2(\pa B_*)}$ with $\eta>0$ small enough.
\end{proofItemize}
\end{proof}

\subsection{Step 3: improved continuity of the second order shape derivative}\label{ssect:continuity}
\label{sec: imprcont}
We introduce here some useful notations.
\begin{defi}
    Let $X,Y$ be two normed vector spaces, and 
        $J: X\to Y$.
    We write
    \[
        J(x)=\om_Y^X(x)
    \]
    to indicate that 
    \[
        \lim_{\norm{x}_X\to 0}{\norm{J(x)}_Y}=0.
    \]
    In particular, when $Y=\R$ we only write $\om_Y^X=\om^X$. Moreover, when $Y=W^{k,p}(B_1;\R^n)$ and $X=W^{j,q}(B_1;\R^n)$ then we write 
    \[
        \om_Y^X=\om_{k,p}^{j,q}.
    \]
    When $Y=L^p(B_1)$ we write $\om_Y^X=\om_{p}^X$.

    Also, when $J:[0,1]\times X\to Y$ depends on the extra parameter $t\in[0,1]$, $J(x)=\omega_Y^X(x)$ is meant in a uniform sense in $t$, namely $\lim_{\|x\|_X\to 0}\sup_t\|J(t,x)\|_Y=0$.
\end{defi}  

The main aim of this section is to prove the following.
\begin{prop}
\label{prop: improvcontin}
   Let $m\in(0,\abs{B_1})$, $q>n$ and $j\in C^1(\R_+)\cap C^2((0,\infty))$ such that 
   \[      
    \red{\lim_{s\to0^+}s^\al j''(s)\in (0,+\infty)}
\]
\red{for some $\alpha\in(-\infty,1)$.} Then
  \[L''(t)=L''(0)+\om^{2,q}(\Phi)\,\norm{\Phi}^2_{L^2(\pa B_*)},\]
  that is to say, for every $\eps>0$ there exists $\eta>0$ such that for every $t\in[0,1]$ and for every $\Phi\in W^{2,q}(B_1;\R^n)$ orthogonal to $\pa B_*$ such that $\norm{\Phi}_{2,q}\le \eta, $ then 
    \[
       \qquad\qquad \abs{L''(t)-L''(0)}\le \eps \norm{\Phi}_{L^2(\pa B_*)}^2,
    \]
    where $L$ is defined in \autoref{cor: lagrangian}.
\end{prop}

As explained in the introduction of the proof of \autoref{prop: eigentostab}, it is not restrictive to assume $\Phi$ to be constant in the normal direction to $\pa B_*$ (in a neighborhood of $\pa B_*$), which allows to apply \autoref{cor: lagrangian}, leading to the following formula:

\begin{equation}
\label{eq: L''Bt}
 L''(t)=\int_{\pa B_*}J_t^\tau \ml(\widehat{w'_t} g+g^2\widehat{\na w_t}\cdot\hnu_t\mr)\,d\Hn+b_t(\Phi,\Phi),
\end{equation}
where for every function $h$ we have denoted by $\widehat{h}=h\circ(\id+t\Phi)$, $J_t^\tau=\jac^{\pa B_*}(\id +t\Phi)$, $g=\Phi\cdot\nu_0$ and
\begin{equation}
\label{eq: bt}
    b_t(\Phi,\Phi)=\int_{\pa B_*} J_t^\tau(\hw_t-w_0)(\hnu_t\cdot\nu_0) g^2\, \widehat{\divv(\nu_0)}\,d\Hn.
\end{equation}
Notice also that we used $\widehat{g}=g$ and $\widehat{\nu_0}=\nu_0$.
To prove \Cref{prop: improvcontin} we need some geometric estimates on $\hnu_t$ and $I_n+tD\Phi$, and estimates on $\hw_t$ and $\widehat{w'_t}$ that will be resumed in the following lemmas. 
\begin{lemma}[{\cite[Lemma 4.3, Lemma 4.7]{DL19},\cite[Lemma 3.7]{P24}}]
\label{lem: geomestim}Let $q>n$ and $\Phi\in W^{2,q}(B_1;\R^n)$. Then
    \begin{gather*}
        (I_n+tD\Phi)^{-1}=I_n+\om_\infty^{2,q}(\Phi) \qquad \qquad
        \det(I_n+tD\Phi)=1+\om_\infty^{2,q}(\Phi) \\[5 pt]
        \hnu_t=\nu_0+\om_\infty^{2,q}(\Phi), \qquad \qquad 
        J^\tau_t=1+\om_\infty^{2,q}(\Phi),
    \end{gather*}
   
\end{lemma}
\begin{lemma}
\label{lem: hupestim}
    Let $m\in(0,\abs{B_1})$, $\Phi\in W^{2,q}(B_1;\R^n)$ with $q>n$. We denote $u_t:=u_{B_*^{t\Phi}}$, and $\hup=u'_t\circ(\id+t\Phi)$. Then there exist constants $C=C(m,n,q)$ and $\de=\de(m,n,q)$ such that if $\norm{\Phi}_{2,q}\le \de$ then
    \begin{equation}
    \label{eq: hupestim}
        \norm{\hup}_{1,2}\le C\norm{\Phi}_{L^2(\pa B_*)}.
    \end{equation}
    Moreover, we have
    \begin{equation}
    \label{eq: hupcont}
        \hup=u'_0+\om_{1,2}^{2,q}(\Phi)\norm{\Phi}_{L^2(\pa B_*)}.
    \end{equation}
\end{lemma}
\begin{proof}
    Let us recall that, by \Cref{prop: u'}, $u'_t$ solves the equation
    \[
          \int_{B_1} \nabla u'_{t}\cdot\nabla \varphi \, dx=\int_{\pa B_*^{t\Phi}} \varphi \,\tphi_t\cdot\nu_t\, d\Hn \qquad \qquad \forall\varphi\in H^1_0({B_1}).
    \]
    In particular, with the change of variables $x=(\id+t\Phi)(y)$, we get
    \begin{equation}
    \label{eq: hupweak}
      \int_{B_1} A_t\nabla \hup\cdot\nabla \vp \, dy=\int_{\pa B_*} J^\tau_t \varphi \,\hg_t\, d\Hn \qquad \qquad \forall\varphi\in H^1_0({B_1}), 
    \end{equation}
    where
    \[
        J_t^\tau:=\jac^{\pa B_*}(\id+t\Phi),\qquad \qquad A_t:= \det(I_n+tD\Phi) (I_n+t D\Phi)^{-1}(I_n+t D\Phi)^{-T}.
    \]
    Choosing $\vp=\hup$, and noticing that by \Cref{lem: geomestim} we have that $A_t$ is uniformly elliptic for small $\norm{\Phi}_{2,q}$, then there exists a positive constant $C$ such that
    \[
        \norm{\na \hup}_2^2\le C\int_{\pa B_*}\abs{\hg_t \hup}\,d\Hn.
    \]
    Using Poincaré inequality, Young inequality, and the embedding $W^{1,2}(B_1)\xhookrightarrow{}L^2(\pa B_*)$, we get for every $\eta>0$
    \[
        \norm{\hup}_{1,2}^2\le C\mleft(\eta\norm{\hg_t}_2^2 + \fr{1}{\eta}\norm{\hup}_{1,2}^2\mright).
    \]
    Formula \eqref{eq: hupestim} follows by choosing a suitable $\eta$, and by recalling $\abs{\hg_t}=\abs{\Phi\cdot \hnu_t}\le\abs{\Phi}$.

Analogously, subtracting the weak equations \eqref{eq: hupweak} solved by $\hup$ and $u'_0$, we have
\begin{equation}
\label{eq: weakhup-up0}
    \int_{B_1} (A_t\nabla \hup-\na u'_0)\cdot\nabla \vp \, dy=\int_{\pa B_*} (J^\tau_t \hg_t-g_0)\varphi \, d\Hn \qquad \qquad \forall\varphi\in H^1_0({B_1}). 
\end{equation}
In the rest of the proof, we evaluate \eqref{eq: weakhup-up0} with the test function $\vp=\hup-u'_0$. We estimate the left-hand side of \eqref{eq: weakhup-up0} by adding and subtracting $A_t\na u'_0\cdot\na \vp$, and using the uniform ellipticity of $A_t$ joint with \Cref{lem: geomestim}, so that for some constant $c>0$
\begin{equation*}
    \int_{B_1} (A_t\nabla \hup-\na u'_0)\cdot\nabla \vp \, dy\ge c\norm{\na(\hup-u'_0)}_2^2-\int_{B_1}\om_\infty^{2,q}(\Phi)\abs*{\na u'_0\cdot\na(\hup-u'_0)}\,dy.
\end{equation*}
Thanks to \eqref{eq: hupestim}, we get
\begin{equation}
    \label{eq: lhshup-up0}
    \int_{B_1} (A_t\nabla \hup-\na u'_0)\cdot\nabla \vp \, dy\ge c\norm{\na(\hup-u'_0)}_2^2-\om^{2,q}(\Phi)\norm{\Phi}_{2,\pa B_*}^{2}.
\end{equation}
On the other hand, we may estimate the right-hand side of \eqref{eq: weakhup-up0} as before and get
\begin{equation}
    \label{eq: rhshup-up0}
    \begin{split}
    \int_{\pa B_*} (J^\tau_t \hg_t-g_0)\varphi \, d\Hn &=\int_{\pa B_*} (1+\om_\infty^{2,q}(\Phi))\Phi\cdot\om_\infty^{2,q}(\Phi)(\hup-u'_0) \, d\Hn \\[7 pt]
    &\le \om^{2,q}(\Phi)\mleft(\eta \norm{\Phi}_{L^2(\pa B_*)}^2+\fr{1}{\eta}\norm{\hup-u'_0}_{1,2}^2\mright)
    \end{split}
\end{equation}
for every $\eta>0$. Joining \eqref{eq: weakhup-up0}, \eqref{eq: lhshup-up0}, \eqref{eq: rhshup-up0}, and the Poincaré inequality, with the right choice of $\eta$ we get
\[
\norm{\hup-u'_0}_{1,2}^2\le \om^{2,q}(\Phi)\norm{\Phi}_{L^2(\pa B_*)}^2.
\]
\end{proof}
\begin{lemma}
\label{lem: hwpestim}
    Let $m\in(0,\abs{B_1})$ and $j\in C^1(\R_+)\cap C^2((0,\infty))$ such that
    \begin{equation}
    \label{eq: lemhwpestimAssumptionj''}
    \red{\lim_{s\to0^+}s^\al j''(s)\in (0,+\infty)}
\end{equation}
\red{for some $\alpha\in(-\infty,1)$.} 
Let $\Phi\in W^{2,q}(B_1;\R^n)$ with $q>n$,  $w_t:=w_{B_*^{t\Phi}}$ the adjoint state defined in \Cref{defi: adjoint}, and  $\hwp=w'_t\circ(\id+t\Phi)$. Then there exist constants $C=C(j,m,n,q)$ and $\de=\de(j,m,n,q)$ such that if $\norm{\Phi}_{2,q}\le \de$ then
    \begin{equation}
    \label{eq: hwpestim}
        \norm{\hwp}_{1,2}\le C\norm{\Phi}_{L^2(\pa B_*)}.
    \end{equation}
    Moreover, we have
    \begin{equation}
    \label{eq: hwpcont}
        \hwp=w'_0+\om_{1,2}^{2,q}(\Phi)\norm{\Phi}_{L^2(\pa B_*)}.
    \end{equation}
\end{lemma}
\begin{proof}
    As in the proof of \Cref{lem: hupestim}, we have
    \begin{equation}
      \int_{B_1} A_t\nabla \hwp\cdot\nabla \vp \, dy=\int_{B_1} \det(I_n+tD\Phi)\,  \red{j''(\hu_t)}\hup \varphi \, dy \qquad \qquad \forall\varphi\in H^1_0({B_1}).
    \end{equation}
    By standard elliptic estimates joint with the continuity of $j'$, the equi-boundedness of $u_t$, and the geometric estimates \Cref{lem: geomestim}, we get
    \[
        \norm{\hwp}_{1,2}\le C\norm{\red{j''(\hu_t)}\hup}_{2}.
    \]
    \red{We claim that $j''(\hu_t)\le C(1-|x|)^{-1}$, so that Hardy's inequality (see ~\cite[Theorem 21.4]{Opic_Kufner_1990}) and } \eqref{eq: hupestim} in \Cref{lem: hupestim} ensure 
    \[
    \red{\norm{\red{j''(\hu_t)}\hup}_{2}\le C\norm{\hup}_{1,2}\le C \norm{\Phi}_{L^2(\pa B_*)}}
    \]
    and as a consequence we get \eqref{eq: hwpestim}. \red{Indeed, from the assumption~\eqref{eq: lemhwpestimAssumptionj''}, we get $j''(s)\le C s^{-1}$ while from the Hopf lemma we know that $\abs{\na u_0}>0$ on $\pa B_1$ and by $C^1$ convergence, also $\abs{\na \hu_t}\ge c>0$ uniformly near $\pa B_1$. In particular, $\hu_t\ge c(1-|x|)^{-1}$ for some possibly smaller $c$ and  $j''(\hu_t)\le C(1-|x|)^{-1}$ as claimed.}

    Analogously, we now estimate the norm of $\hwp-w'_0$ rewriting the equation as
    \begin{equation}
    \label{eq: weakhwp-wp0}
        \underbracket{\int_{B_1} (A_t\nabla \hwp-\na w'_0)\cdot\nabla \vp \, dy}_{I_1}=\underbracket{\int_{B_1} (J_t\, \red{j''}(\hu_t)\hup-\red{j''}(u_0)u'_0)\varphi \, dy}_{I_2} \qquad \qquad \forall\varphi\in H^1_0({B_1}). 
    \end{equation}
As done in the proof of \Cref{lem: hupestim}, we evaluate \eqref{eq: weakhwp-wp0} with the test function $\vp=\hwp-w'_0$ and get
\begin{equation}
    \label{eq: lhshwp-wp0}
   I_1\ge c\norm{\na(\hwp-w'_0)}_2^2-\om^{2,q}(\Phi)\norm{\Phi}_{L^2(\pa B_*)}^{2}.
\end{equation}
We also notice that by the shape differentiability of $\hu_t$ (see the proof of \Cref{prop: u'}) we have
\[
\hu_t=u_0+\om_\infty^{2,q}(\Phi).
\]
Also, \red{we claim that 
\[
\fr{j''(\hu_t)}{j''(u_0)}=1+\om_\infty^{2,q}(\Phi).
\]
Indeed, we start by noticing that the assumption~\eqref{eq: lemhwpestimAssumptionj''} ensures that $s^\alpha j''(s)$ is uniformly continuous on every compact $[0,M]$, with $s^\al j''(s)\ge c>0$, so that
\[
\frac{s^\al j''(s)}{\tau^\al j''(\tau)}=1+\om(\abs{s-\tau}),
\]
uniformly in $s$ and $\tau$. We then observe that 
\[
\fr{j''(\hu_t)}{j''(u_0)}=\ml(1+\om_\infty^{2,q}(\Phi)\mr)\fr{u_0^\al}{\hu_t^\al} 
\] 
because by elliptic estimates $\hu_t-u_0=\om_{\infty}^{2,q}(\Phi)$. Finally, we notice that since $\abs{\na u_0}\ge c>0$ on $\pa B_1$, we can find uniform constants $c,C>0$ such that $\hu_t\ge c(1-|x|)$, and 
\[
\abs{\hu_t-u_0}\le C\norm{\na\hu_t-\na u_0}_\infty (1-|x|).
\]
Therefore, 
\[
\fr{u_0}{\hu_t}=1+\om_\infty^{2,q}(\Phi),
\]
and joining the equalities we get the claim. Therefore, \red{similarly to the proof of} \Cref{lem: hupestim}, and using $j''(u_0)\le C(1-|x|)^{-1}$ with the Hardy inequality $\norm{\vp/(1-|x|)}_2\le C\norm{\vp}_{1,2}$ (see~\cite[Theorem 21.4]{Opic_Kufner_1990}), we get}
\begin{equation}
    \label{eq: rhshwp-wp0}
    \begin{split}
    I_2 &\le C\int_{B_1} \red{j''(u_0)}\abs*{\big(1+\om_\infty^{2,q}(\Phi)\big)\hup-u'_0)}\abs*{\hwp-w'_0}\, dy \\[7 pt]
    &\le C\int_{B_1} j''(u_0)\abs*{\hup-u'_0}\abs*{\hwp-w'_0}\, dy + \om^{2,q}(\Phi)\int_{B_1}\red{j''(u_0)}\abs*{\hup}\abs*{\hwp-w'_0}\,dy \\[7 pt]
    &\le C \om^{2,q}(\Phi)\mleft(\eta \norm{\Phi}_{2,\pa B_*}^2+\fr{1}{\eta}\norm{\hwp-w'_0}_{1,2}^2\mright)
    \end{split}
\end{equation}
for every $\eta>0$ and a suitable constant $C>0$. Joining \eqref{eq: weakhwp-wp0}, \eqref{eq: lhshwp-wp0}, \eqref{eq: rhshwp-wp0}, and the Poincaré inequality, with the right choice of $\eta$ we get
\[
\norm{\hwp-w'_0}_{1,2}^2\le \om^{2,q}(\Phi)\norm{\Phi}_{2,\pa B_*}^2,
\]
which implies \eqref{eq: hwpcont}.
\end{proof}

\begin{proof}[Proof of \Cref{prop: improvcontin}]

    \textbf{Step 1:} we first prove 
    \begin{equation}
    \label{eq: step1improvcont}
        b_t(\Phi,\Phi)=\om^{2,q}(\Phi)\norm{\Phi}_{L^2(\pa B_*)}^{2}.
    \end{equation}
    Since $\nu_0(x)=\abs{x}^{-1}x$, then we can rewrite \eqref{eq: bt} as
    \[
        b_t(\Phi,\Phi)=\int_{\pa B_*}J_t^\tau(\hw_t-w_0)(\hnu_t\cdot\nu_0)g^2\,\fr{n-1}{r_*+tg}\,d\Hn.
    \] 
    By \Cref{lem: geomestim} there exists a constant $C>0$ such that if $\norm{\Phi}_{2,q}$ is small enough, then 
    \[
    \abs{J_t^\tau}+\abs{\hnu_t\cdot\nu_0}+\fr{1}{r_*+tg}\leq C.
    \]
    Moreover, by shape differentiability of $\hw_t$ proved in \Cref{prop: firstshapeder}, we have
    \begin{equation}
    \label{eq: W1infhw}
        \hw_t=w_0+\om_{1,\infty}^{2,q}(\Phi).
    \end{equation}
    These estimates yield
    \begin{equation} 
    \label{eq: estimatebt}
        b_t(\Phi,\Phi)=\om^{2,q}(\Phi)\norm{\Phi}_{L^2(\pa B_*)}^2.
    \end{equation}
    \textbf{Step 2:} we now prove that
    \[
        L''(t)-b_t(\Phi,\Phi)=L''(0) + \om^{2,q}(\Phi)\norm{\Phi}_{L^2(\pa B_*)}^2.
    \]
    By \eqref{eq: L''Bt}
    \[
        L''(t)-b_t(\Phi,\Phi)=\int_{\pa B_*}J_t^\tau \ml(\widehat{w'_t} g+g^2\widehat{\na w_t}\cdot\nu_0\mr)(\hnu_t\cdot \nu_0)\,d\Hn
   =L''(0)+I_1+I_2,
    \]
    where
    \[
        I_1= \int_{\pa B_*}(J^\tau_t\widehat{w_t'}(\hnu_t\cdot \nu_0)- w'_0) g\, d\Hn, \] 
        \[
        I_2=\int_{\pa B_*}g^2\mleft(J_t^\tau(\widehat{\na w_t}\cdot\nu_0)(\hnu_t\cdot \nu_0)-(\na w_0\cdot\nu_0)\mright)\,d\Hn.
    \]
    To estimate $I_1$, we use \Cref{lem: geomestim}, and \Cref{lem: hwpestim} to get, dropping the dependence on $\Phi$ inside the infinitesimal notation (i.e. $\om_X^Y=\om_X^Y(\Phi)$),
    \[
    \begin{split} 
        I_1 &=\int_{\pa B_*} g\ml(\ml(1+\om_\infty^{2,q}\mr)\ml(w_0'+\om_{2,\pa B_*}^{2,q}\norm{\Phi}_{L^2(\pa B_*)}\mr)-w_0'\mr)\, d\Hn \\[7 pt] 
        &=\om^{2,q} \norm{\Phi}_{L^2(\pa B_*)}^2,
    \end{split}
    \]
    where we used H\"older inequality, $\norm{g}_{2}=\norm{\Phi}_{L^2(\pa B_*)}$, and \eqref{eq: hwpestim}.

    Finally, to estimate $I_2$ we use again \eqref{eq: W1infhw} and we notice that
    \begin{equation}
    \label{eq: hatnablaw}
        \widehat{\na w_t}=(I_n+tD\Phi)^{-T}\na \hw_t.
    \end{equation}
    Hence, using \eqref{eq: hatnablaw}, \eqref{eq: W1infhw}, and \Cref{lem: geomestim}, we get
    \[
        I_2=\int_{\pa B_*}g^2\Big((I_n+\om_\infty^{2.p})(\na w_0\cdot\nu_0+\om_\infty^{2,q})(1+\om_\infty^{2,q})-(\na w_0\cdot \nu_0)\Big)\, d\Hn=\om^{2,q}\norm{\Phi}_{L^2(\pa B_*)}^2.
    \]
\end{proof}

\subsection{Conclusion to the proof of \Cref{thm: fuglede}}
\label{sec:conclusion}
We are now in position to prove \Cref{thm: fuglede}: let $m,q,j$ as in the statement of the result. Let also $\Phi\in W^{2,q}(B_1;\R^n)$ orthogonal to $\pa B_*$ such that $|B_*^\Phi|=|B_*|$. We apply \Cref{cor: lagrangian} providing the Lagrangian $\Ll_\tau$ for some $\tau\in\R$ and 
    \[
        \forall t\in[0,1], \;L(t)=\Ll_\tau(B_*^{t\Phi}).
    \]
    By \Cref{prop: eigentostab} there exists $c_1>0$ such that
    \begin{equation}
    \label{eq: stabilityL''0}
        \Ll''_\tau(B_*)[\Phi,\Phi]=L''(0)\le -c_1\norm{\Phi\cdot\nu_0}_{L^2(\pa B_*)}^2=-c_1\norm{\Phi}_{L^2(\pa B_*)}^2.
    \end{equation}
    By \Cref{prop: improvcontin}, we get the existence of $\eta>0$ such that if $\|\Phi\|_{2,q}\leq \eta$,
    \begin{equation}
    \label{eq: stabilityL''t}
       \forall t\in[0,1],\qquad L''(t)\le L''(0)+ \fr{c_1}{2} \norm{\Phi}_{L^2(\pa B_*)}^2.
    \end{equation}
    Joining \eqref{eq: stabilityL''0}, \eqref{eq: stabilityL''t}, and the optimality condition $L'(0)=0$, we obtain for some $t_0\in(0,1)$
    \[
    \begin{split}
         \Ll_\tau(B_*)-\Ll_\tau(B_*^{\Phi})=L(0)-L(1) =-\frac{L''(t_0)}{2}\ge \fr{c_1}{4}\norm{\Phi}_{L^2(B_*)}^2.
    \end{split}
    \]
    The result now follows by noticing that, since $\Phi$ is orthogonal to $\pa B_*$, and $\norm{\Phi}_\infty$ is arbitrarily small, with $g=(\Phi\cdot\nu_0)_{|\pa B_*}$,
    \[
        \abs{B_*\De B_*^\Phi}= \fr{1}{n}\int_{\pa B_*}\abs*{(1+g)^n-1}\,d\Hn\le C(n) \norm{g}_{L^1(\pa B_*)}\le C(n,m)\norm{\Phi}_{L^2(\pa B_*)}.
    \]

\section{The case $\J(E)=\norm{u_E}_\infty$}
\label{sect:Linf}
\red{In this section, we deal with the case $p=\infty$ from \autoref{thm: maintheorem} that is a consequence of the following result:}
\begin{thm}
    \label{thm: maintheoremInfinity} 
    Let $n\ge 1$ and $m\in(0,\abs{B_1})$. Then there exists a positive constant $c=c(m,n)$ such that for every $V\in\M_m$ we have
    \[
        \norm{u_{B_*}}_\infty-\norm{u_V}_\infty\ge c\norm{V-\chi_{B_*}}_1^2.
    \]
    where $B_*$ is the centered ball of volume $m$.
\end{thm}
\red{The strategy is very similar to the proof of \autoref{thm: pnorm}, so we advise the reader to be familiar with the strategy described in Sections \ref{fugledeImpliesStability} and \ref{sect:fuglede}. We will insist on the most significant modifications to the proofs.}
In this section we denote $\J(V)=\norm{u_V}_\infty$. 

\subsection{Proof of \autoref{thm: maintheoremInfinity} from a Fuglede-type result}\label{ssect: fugledeinfinity}
As done for the case $1<p<+\infty$, we start by showing that \autoref{thm: maintheoremInfinity} will follow from the following Fuglede-type result:.
\begin{thm}
    \label{thm: fugledeInfinity}
       Let $m\in(0,\abs{B_1})$, $q>n$. Then there exist positive constants $c=c(m,n)$, $\eta=\eta(m,n)$ such that for every $\Phi\in W^{2,q}(B_1,\R^n)$ orthogonal to $\pa B_*$ with
        \[
            \norm{\Phi}_{2,q}\le \eta, \qquad \qquad \abs{B_*^\Phi}=\abs{B_*}=m,,
        \]
        we have
        \[
             \J(B_*)-\J(B_*^{\Phi})\ge c \abs{B_*\De B_*^{\Phi}}^2.
        \]
\end{thm}
We postpone the proof of \Cref{thm: fugledeInfinity} to the next sections, and we show that \Cref{thm: maintheoremInfinity} holds. For $r\in(0,1)$, let
\[
    \ze(r) = \begin{dcases}
        -\frac{1}{2\pi}\ln(r) &n=2, \\ 
        \frac{1}{n(n-2)\om_n} r^{2-n} &n\ne 2,
    \end{dcases}
\]
be the fundamental solution in $\R^n$ to the Laplace equation. We identify $\ze(x)$ with $\ze(\abs x)$ and we recall the Green's function for the ball $B_1$ given by
\[
 G(x,y) = \ze(y-x) - \ze (\abs{x}(y-\tx)), \qquad \qquad \tx =\frac{x}{\abs{x}^2},   
\]
and
\[
    G(0,y) = \ze(y)-\ze(1).
\]
Often we will use the notation $G_x(y)=G(x,y)$. Let $E\sbs B_1$ and let $x_E$ be a maximum point of $u_E$. Then\
\[
    \J(E)=\norm{u_E}_\infty  = u_E(x_E) = \int_{E} G(x_E,y)\,dy.
\]
The function $G(x_E,\cdot)$ will play the role of the adjoint state $w_E$ in the case $1<p<+\infty$, as we will see later. 
\begin{proof}[Proof of \Cref{thm: maintheoremInfinity}]
   \begin{proofItemize}
	 \item Similarly to \Cref{prop: exist}, we can prove that $V\in L^\infty(B_1)\mapsto \J(V)$ is convex (not strictly) and weakly-$*$ continuous in $L^\infty(B_1)$. Therefore, there exists a bang-bang maximizer of $\J$ both in $\M_m$ and in $\M_m^\de$ for every $\de> 0$. Moreover, the maximizer in $\M_m$ is the ball and it is unique. Indeed, $B_*$ is a maximizer by Talenti's inequality and the fact that there exists a bang-bang maximizer. About uniqueness, let $f\in \M_m$ be a maximizer for $\J$ and notice that, using the Talenti's inequality
     \[
        \int_{B_1}G(0,y)f^\sh(y)\,dy  = \J(f^\sh) \ge \J(f) = \J(B_*) =  \int_{B_1}G(0,y)\chi_{B_*}(y)\,dy;
     \]
     by the rigidity of the bathtub principle, since $\J(B_*)\ge\J(f^\sh)$, the equality implies that $f^\sh = \chi_{B_*}$ and then $f=\chi_E$ for some measurable set $E$.
    Therefore, $E=B_*$ because $B_*$ is the unique maximizer of $\J$ among characteristic functions because of the rigidity of the Talenti's inequality for the $L^\infty$ norm (see \cite[Corollary 1]{ALT86}).     
     Therefore, \Cref{prop: smalldelta} also holds for $\J(V)=\norm{u_V}_\infty$, and proving the theorem is equivalent to prove that 
    \[
    \liminf_{\delta\to 0}\inf_{V\in\M_m^\delta}\frac{\J(B_*)-\J(V)}{\delta^2}>0.
    \]

    \item Let us assume by contradiction that 
    \[
        \lim_{\de\to 0}\fr{\J(B_*)-\J(E_\de)}{\de^2}=0,
    \]
    where $E_\de$ is the maximizer of $\J$ in $\M_m^\de$.
    Since for every measurable set $E\sbse B_1$,
    \[
        \J(E)=\int_{E}G(x_E,y)\,dx, \qquad \qquad x_E\in \argmax u_E,
    \]
    we can reproduce the proof of \Cref{thm: maintheorem} from \autoref{ssect: mainproof}, provided that $x_\de=x_{E_\de}$ is uniquely determined and  $x_\de$ converges to $0=x_{B_*}$. Indeed, in that case, since $G\in C^\infty(B_1\times B_1\sm \diag(B_1\times B_1))$, we have that for every fixed small $\eps>0$, the functions $G_{\de}$ converge to $G_0$ in $C^2(B_1\sm B_\eps)$. Therefore, we have non-degeneracy of $\na G_\de$ and the result will follow by retracing the proof of \Cref{thm: maintheorem}: with similar computations as in \autoref{ssect: mainproof}, we obtain
    \begin{itemize}
        \item for every $\de$ small there exists a level set $\tEde$ of $G_\de$ of measure $m$;
        \item by the quantitative bathtub principle applied to $G_\de$,
        \[\J(\tEde)-\J(E_\de)\geq u_{\widetilde{E}_\de}(x_\de)-u_{E_\de}(x_\de)=\int_{B_1}  G(x_\de,\cdot)(\chi_{\widetilde{E}_\de}-\chi_{E_\de})\geq c|E_\de\Delta \tEde|^2\]    
        \item by \autoref{lem: convdef} $\tEde$ are smooth deformations of $B_*$, leading to a contradiction with \autoref{thm: fugledeInfinity}.
    \end{itemize}

   \item We let $u_\de=u_{E_\de}$. We now show that $x_\de$ is uniquely defined and it converges to $x_{B_*}=0$. Since $\abs{\na u_0}=0$ only in the origin, we have that any critical point of \(u_\de \) has to be uniformly close to the origin. Thanks to the Schauder's estimates  we have that for every $K\ssubset B_1\sm \pa B_*$ 
    \[
        u_\de \xrightarrow{C^2(K)} u_0,
    \]
    and since $u_0$ has a unique maximum, and by explicit computations $-D^2 u_0(0)$ is positively definite, we apply the implicit function theorem to the equation $\na u_\de(x_\de)=0$, so that for small $\de$, the functions $u_\de$ also have a unique maximum $x_\de$ converging to 0.
\end{proofItemize}
\end{proof}
\subsection{Computation of shape derivative}
We will now study optimality conditions when the objective functional is $E\mapsto\norm{u_E}_\infty$. A similar study was also performed in the paper \cite{HenLucPhil2018}, where the authors deal with the shape derivative of the $L^\infty$ norm of the torsion function of a set $\Om$; our analysis goes further as we compute second order derivatives (which requires a computation of the derivative of the maximum point).

We fix $q>n$. As we saw in the proof of \Cref{thm: maintheoremInfinity}, for $\Phi\in W^{2,q}(B_1;\R^n)$ small, $u_\Phi$ has a unique maximum point $x_\Phi$, and we now show that we can differentiate $x_\Phi$ with respect to $\Phi$. 
\begin{prop}[Shape differentiability of $x_\Phi$]
    \label{prop: x'infinity}
    The application
    \[
        \Phi\in W^{2,q}(B_1;\R^n)\longmapsto x_\Phi=\argmax u_{B_*^\Phi}\in \R^n
    \]
    is of class $C^1$ in a neighborhood of 0.
    In particular, if $\Phi\in W^{2,q}(B_1;\R^n)$ is small enough and we denote by $x_t:=x_{t\Phi}$, then we have that for every $t\in[0,1]$, 
    \[
    x'_t = -\red{\ml(D^2 u_t(x_t)\mr)^{-1}} \na u'_t(x_t).
    \]
\end{prop}
\begin{proof}
    Since $u_\Phi$ depends only on the values of $\Phi$ on $\pa B_*$, without loss of generality we may assume that $\supp(\Phi)\sbs B_1\sm B_{r_*/2}$. Let us recall that $x_\Phi$ is uniquely defined as the solution to the equation
    \begin{equation}
        \label{eq: implicitxphi}
        \na u_\Phi(x_\Phi) = 0.  
    \end{equation}
    As shown in the proof of \Cref{thm: maintheoremInfinity}, we know that $x_\Phi$ is converging to 0 as $\Phi$ goes to 0, and we will assume to have $x_\Phi\in B_{r_*/2}$. 
    Since for small $\Phi$ we have $\chi_{B_*^\Phi}=1$ in $B_{r_*/2}$, then using the Schauder's estimate $\norm{u_\Phi-u_0}_{C^2(B_{r_*/2})}\le C_\eps\norm{u_\Phi-u_0}_\infty$ we can easily adapt the proof of \Cref{prop: u'} to show that 
    \[
        \Phi\in W^{2,q}(B_1;\R^n)\cap H^1_0(B_1\sm B_{r_*/2})\longmapsto u_\Phi\in W^{1,q}_0(B_1) \cap C^2(B_{r_*/2})
    \]
    is of class $C^1$. The conclusion follows by the implicit function theorem applied to equation~\eqref{eq: implicitxphi}.
\end{proof}

\begin{prop}[Representation formula for $u'_t$]\label{prop: u'infinity}
Let $\Phi\in W^{2,q}(B_1;\R^n)$, let $u_t:=u_{t\Phi}$. Then for every $x\in B_1\sm \pa E_t$ we have 
    \begin{equation}
    \label{eq: u'infty}
        u'_t(x)=\int_{\pa E_t} G_{x} \,(\tphi_t\cdot\nu_t)\,d\Hn.
    \end{equation}
    where $\widetilde{\Phi}_t=\Phi\circ(\id+t\Phi)^{-1}$.
\end{prop}
\begin{proof}
    We recall that by \Cref{prop: u'}, 
    \[
        -\De u'_t = (\tphi_t\cdot\nu)\,\mrestr{d\Hn}{\pa B_*^{t\Phi}}.
    \]
    The result then follows by writing $u'_t$ with the Green's representation formula.
\end{proof}
In the following we denote by $\na_x G_{x_t}(y) = \na_x G(x_t,y)$ and  $\na_y G_{x_t}(y) = \na_y G(x_t,y)$. Moreover, we let $G'_t(y)=\na_x G_{x_t}(y)\cdot x'_t$.
\begin{prop}
    The application
    \[
        \Phi\in W^{2,q}(B_1;{\R^n})\longmapsto \J(B_*^\Phi)\in \R
    \]
    is of class $C^2$ in a neighborhood of 0. If $\Phi\in W^{2,q}(B_1;\R^n)$ is small enough and $J(t)=\J(B_*^{t\Phi})$, then for $t\in[0,1]$,
\begin{equation} 
\label{eq: J'Infinity}
\begin{split}
    J'(t)=\int_{\partial B_*^{t\Phi}} G_{x_t}\, (\tphi_t\cdot\nu_t)\,d\Hn,
\end{split}
\end{equation}
where $\nu_t=\nu_{E^{t\Phi}}$ is the outer unit normal to $B_*^{t\Phi}$, and $\tphi_t=\Phi\circ(\id+t\Phi)^{-1}$. Moreover, letting $g_t=\tphi_t\cdot\nu_t$,
\begin{equation} 
\label{eq: J''Infinity}
\begin{split}
    J''(t)=\int_{\pa B_*^{t\Phi}}\ml(g_t G'_t+ g_t\na_y G_{x_t}\cdot\tphi_t\mr)\,d\Hn(y)+a_t(\Phi,\Phi),
\end{split}
\end{equation}
where
\[
a_t(\Phi,\Phi)=\int_{\pa B_*^{t\Phi}}G_{x_t}\ml(g_t \divv(\tphi_t)-{((D\tphi_t)\tphi_t)}\cdot\nu_t\mr)d\Hn.
\]
\end{prop}

\begin{proof}
    In the following we denote by $E_t= B_*^{t\Phi}$. We first notice that $J(t) = u_t(x_t)$, so that by \Cref{prop: u'infinity}
    \[
        J'(t) = u'_t(x_t) + \na u_t(x_t)\cdot x'_t.
    \]
    Since $x_t$ is the maximum point for $u_t$, then $\na u_t(x_t)=0$. Using the representation formula for $u'_t$, we get 
    \[
        J'(t)=\int_{\pa E_t}^{}G_{x_t}(\tPhi_t\cdot\nu_t)\,dx = -\int_{B_1\sm E_t}\divv(G_{x_t}\tPhi_t)\,dy.
    \] 
    Since $G_{x_t}$ is $C^\infty(B_1\sm B_{r_*/2})$, we can apply the Hadamard's formula to conclude the proof.  
\end{proof}
As done in \Cref{ssect:shapederivatives}, the previous results lead to the following one about the Lagrangian.
\begin{cor}\label{cor: lagrangianInfinity}
    For $\tau\in\R$ and $E\subset B_1$, we define 
     \[
            \Ll_\tau(E):=\J(E)+\tau\abs{E}.
        \]
    For $m\in(0,|B_1|)$ and $E=B_*$ the centered ball of volume $m$, we set
        \begin{equation} 
        \label{eq: tauInfinity}
            \tau=-\restr{G_0}{\partial B_*}.
        \end{equation}
        Then $\Phi\in W^{2,q}(B_1;\R^n)\mapsto \Ll_\tau(B_*^\Phi)$ is of class $C^2$ near 0, and
        \begin{enumerate}[label=(\roman*)]
        \item $\Ll_\tau'(B_*)\equiv 0$
        \item for every $\Phi\in W^{2,q}(B_1;\R^n)$ small enough such that $\Phi$ is normal on $\partial B_*$ and $\widetilde{\Phi}_t=\Phi$ for $t\in[0,1]$, if we denote $L(t)=\Ll_\tau(B_*^{t\Phi})$ then for every $t$,
        \begin{equation}
        \label{eq: L''tInfinitys}
            \begin{split}
	L''(t)=&\int_{\pa B_*^{t\Phi}}\ml(G'_t g+g^2\fr{\pa G_{x_t}}{\pa \nu_0}\mr)(\nu_t\cdot\nu_0)\,d\Hn(y) \\ 
    &+\int_{\pa B_*^{t\Phi}}(w_t-G_0(r_*))(\nu_t\cdot\nu_0)g^2\divv(\nu_0)\,d\Hn(y),
\end{split}   
    \end{equation}
        where $G_0(r_*)=G_0(y)$ for any $y\in\pa B_*$ and $g=\Phi\cdot\nu_0$ (recall that we consider an extension of $\nu_0$ through the projection onto $\pa B_*$).
        \end{enumerate}
        \end{cor}
\subsection{Coercivity in 0}
As done for \Cref{thm: fuglede}, we now prove that the Lagrangian is coercive. 
\begin{prop}
    \label{prop: eigentostabInfinity}
    Let $m\in(0,\abs{B_1})$, $B_*$ the centered ball of volume $m$. Then there exist positive constants $c=c(n,m), \eta=\eta(n,m)$ such that for every $\Phi\in W^{2,q}(B_1;\R^n)$ \red{such that $|B_*^\Phi|=|B_*|$} and $\norm{\Phi}_{\infty}<\eta$ we have
    \begin{equation*}
       \Ll_\tau''(B_*)[\Phi,\Phi]\le - c\norm{\Phi\cdot \nu_0}_{L^2(\pa B_*)}^2,
    \end{equation*}
    where $\Ll_\tau$ is the Lagrangian defined in \Cref{cor: lagrangianInfinity}.
\end{prop}
\begin{proof}
    Let $\Phi$ be as in \Cref{cor: lagrangianInfinity} (as explained in \autoref{ssect:coercivity}, this is not restrictive), so that 
    \begin{equation}
        \label{eq:l''0Infinity}\Ll_\tau''(B_*)[\Phi,\Phi]=L''(0) = \int_{\pa B_*} \ml(g \,\ml(\na_x G_0\cdot x'_0\mr) + g^2 \,\ml(\na_y G_0\cdot\frac{y}{\abs{y}}\mr)\mr)\,d\Hn(y).
    \end{equation}
    By \Cref{prop: x'infinity}
    \[
        x'_0 = \frac{1}{n}\int_{\pa B_*} \na_x G(0,z)\, g(z)\,d\Hn,
    \]
    where we used again that $D^2 u_0(0)=-1/n\, I_n$ and \Cref{prop: u'infinity}.
    By the definition of $G$ and its symmetries, we may write 
    \[
        \na_x G(0,y) =  \na_y G(y,0) = - (\ze'(r_*)-\ze'(1)r_*)\fr{y}{\abs{y}}, \qquad \qquad \na_y G(0,y) = \ze'(r_*)\frac{y}{\abs{y}}.
    \]
        If $n=1$, then the constraint $\abs{B_*^\Phi}=\abs{B_*}$ gives that $g$ has to be odd. Therefore, direct computations give $\na_xG_0(y)=\frac{1-r_*}{2}\sgn(y)$, and $x'_0=(1-r_*)g(r_*)$, so that~\eqref{eq:l''0Infinity} reads
        \[
            L''(0) = -\ml(1-(1-r_*)^2\mr) g(r_*)^2,
        \]
        which concludes the proof. Let us now assume $n\ge2$,
   and let us decompose $g$ in spherical harmonics $g=\sum_{k,m}\al_{k,m}Y_{k,m}$, recalling that
    \[
        Y_{1,j}(z) =  \sqrt{\frac{n}{ P(B_1)}}\frac{z_j}{\abs{z}}, \qquad \qquad \al_{1,j}=\frac{1}{\|Y_{1,j}\|^2_{L^2(\pa B_*)}}\int_{\pa B_*}g(z)Y_{1,j}\,d\Hn.
    \]
    We notice that we have 
    \[
        \na_x G_0 \cdot x'_0= \frac{P(B_1)}{n^2} (\ze'(r_*)-\ze'(1)r_*)^2\sum_{j=1}^n\al_{1,j}\|Y_{1,j}\|^2_{L^2(\pa B_*)}Y_{1,j}, \qquad \qquad \na_y G_0\cdot \frac{y}{\abs{y}} = \ze'(r_*)
    \]
    Hence, substituting in \eqref{eq:l''0Infinity}, we obtain
    \[
        L''(0) = \sum_{j=1}^{n} \al_{1,j}^2 \|Y_{1,j}\|^4_{L^2(\pa B_*)}\, \frac{P(B_1)}{n^2}\ml(\ze'(r_*)-r_*\ze'(1)\mr)^2+ \sum_{k=1}^{+\infty}\sum_{j=1}^{M(k)}\al_{k,j}^2\|Y_{k,j}\|^2_{L^2(\pa B_*)}\ze'(r_*).
    \]
    Using $\ze'(r_*)=-(P(B_*))^{-1}$ and $P(B_1)\|Y_{1,j}\|^2_{L^2(\pa B_*)}=P(B_*)$, we rewrite
    \[
        \begin{split}
	L''(0) &=  \sum_{j=1}^{n} \al_{1,j}^2 \|Y_{1,j}\|^2_{L^2(\pa B_*)}\ze'(r_*)\ml(\frac{\ml(1-r_*^n\mr)^2}{n^2}-1\mr) +\sum_{k=2}^{+\infty}\sum_{j=1}^{M(k)}\al_{k,j}^2\|Y_{k,j}\|^2_{L^2(\pa B_*)}\ze'(r_*) \\ 
    &\le  -\frac{1}{P(B_*)} \ml(1-\frac{\ml(1-r_*^n\mr)^2}{n^2}\mr) \norm{g}_2^2,
\end{split}
    \]
    where we used that $\norm{g}_2^2=\sum_{k,j}\al_{k,j}^2\|Y_{1,j}\|^2_{L^2(\pa B_*)}$.
\end{proof}
\subsection{Improved continuity}
Also in this case, we need to show the following improved continuity result.
\begin{prop}
    \label{prop: improvcontinInfinity}
       Let $m\in(0,\abs{B_1})$ and $q>n$. Then
        \[
            L''(t)=L''(0)+ \om^{2,q}(\Phi)\,\norm{\Phi}^2_{L^2(\pa B_*)},
        \]
        where $L(t):=\Ll_\tau(B_*^{t\Phi})$, i.e. for every $\eps>0$ there exists $\eta>0$ such that 
        \[\forall \Phi\in W^{2,q}(B_1;\R^n)\textrm{ with }
            \norm{\Phi}_{2,q}\le \eta,
        \qquad\forall t\in[0,1],\qquad
            \abs{L''(t)-L''(0)}\le \eps \norm{\Phi}_{L^2(\pa B_*)}^2.
        \]
    \end{prop}
    \begin{proof}
        Since the expression of  $L''(t)$ computed in \Cref{cor: lagrangianInfinity} is quite similar to the one computed in \Cref{cor: lagrangian}, we can perform a proof analogous to the one of \Cref{prop: improvcontin}. Thanks to \Cref{prop: u'infinity} and the smoothness of $G(\cdot,y)$ far from $y$, we know that $\Phi \mapsto G_{x_\Phi}\in W^{2,q}(B_1\sm B_{r_*/2})$ is of class $C^1$ in a neighborhood of 0, and in particular 
        \[
        \widehat{G_{x_t}}=G_0+\om^{2,q}_{W^{2,q}(B_1\sm B_{r_*/2})}(\Phi),    
        \]
        where we recall that for every function $h$ we denoted $\widehat{h}=h\circ(\id+t\Phi)$. It remains to prove that, letting $G'_t=\na_x G_{x_t}\cdot x'_t$,
        \begin{equation}
            \label{eq: improvedContGt}
            \widehat{G'_t} = {G'_0}+\om_{L^\infty(\pa B_*)}^{2,q}\norm{\Phi}_{L^2(\pa B_*)}, \qquad \qquad \norm{G'_0}_{L^\infty(\pa B_*)}\le C\norm{\Phi}_{L^2(\pa B_*)}.
        \end{equation}
        Indeed, we have that 
        \begin{equation}
            \label{eq: G't}
            G'_t = \na_x G_{x_t}\cdot D^2 u_t(x_t)\na u'_t(x_t).
        \end{equation}
        Since $G\in C^\infty (B_1\times B_1\sm\diag(B_1\times B_1))$ and since we may assume $\supp\Phi\ssubset B_1\sm B_{r_*/2}$ (as also done in \Cref{prop: u'infinity}) we get
        \begin{equation}
            \label{eq: improvContNablaG}
            \na_x G(x_t,(\id+t\Phi)^{-1}(y)) = \na_x G(0,y) + \om_{L^{\infty}(\pa B_*)}^{2,q}(\Phi),
        \end{equation}
        where we also used the continuity of $\Phi \mapsto x_\Phi$. By Schauder's estimates we get 
        \begin{equation}
            \label{eq: improvContDsquaredu}
        D^2u_t(x_t)=D^2 u_0(0)+\om^{2,q}(\Phi),
        \end{equation} 
        and by the representation formula
        \[
            \na u'_t(x) = \int_{\pa B_*}\na_x G(x,y)g(\nu_0\cdot\nu_t)\,d\Hn
        \]
        we finally obtain
        \begin{equation}\label{eq:improvContNablau'}
            \abs{\na u'_t(x_t)}\le C\norm{g}_{L^2(\pa B_*)}.
        \end{equation}
        Joining~\eqref{eq: G't},~\eqref{eq: improvContNablaG},~\eqref{eq: improvContDsquaredu}, and~\eqref{eq:improvContNablau'} we obtain~\eqref{eq: improvedContGt}. The proof now follows by the same computations in \Cref{prop: improvcontin}.
    \end{proof}
    \begin{proof}[Proof of \Cref{thm: fugledeInfinity}]
        The proof is the same as \Cref{sec:conclusion}.
    \end{proof}
\appendix
\section{Appendix: About the convergence of level sets}

In this section, we add some details to \cite[After formula (68)]{M21} or \cite[Formula (Def)]{MRB22} where the authors deduce a strong convergence of level sets from the convergence of functions, as used in \Cref{ssect: mainproof}. We felt the proofs to be a bit elliptic, and we decided to expand them in this section.
\begin{lemma}
\label{lem: constmeaslevel}
    Let $\Omega\subset\R^n$ be a connected and bounded open set,  $(u_k)$ a sequence of functions in $L^\infty(\Om)$ and $u\in C^0(\Om)\cap L^{\infty}(\Om)$ such that  $\abs{\Set{u=\essinf_\Om u}}=0$. For $t_k,t_*\in \R$ we assume
    \[
    \forall k\in \N, \;\;\abs{\set{u_k>t_k}}=\abs{\set{u>t_*}}\in(0,\abs{\Om}),\qquad\textrm{ and }\qquad        u_k\xrightarrow{L^\infty(\Om)}u.
    \]
    Then $t_*=\lim_k t_k$.
\end{lemma}
\begin{proof}
    We define $m=|\{u>t_*\}|=|\{u_k>t_k\}|$. Since functions $u_k$ converge uniformly, they are equi-bounded. In particular, this ensures that the sequence $t_k$ is bounded in $\R$, otherwise we would either get for large $k$ that $t_k>\sup_j\norm{u_j}_\infty$, and $\abs{\set{u_k>t_k}}=0$, or $t_k\le \inf_j\essinf u_j$ and $\abs{\set{u_k>t_k}}=|\Om|$. Therefore, up to passing to a subsequence, we can define $\bar{t}=\lim_k t_k$, and let $\eps>0$. By $L^\infty$ convergence we get that for large $k$
    \[
        \set{u>\bar{t}+2\eps}\subseteq \set{u_k>\bar{t}+\eps}\subseteq \set{u_k>t_k},
    \]
    and therefore defining $\mu(t)=|\{u>t\}|$ (this differs from \Cref{def: mu} as we did not assume $u$ to be non-negative), we get
    \[
        \mu(\bar{t}+2\eps)\le m.
    \]
    Analogously 
    \[
        \set{u>\bar{t}-2\eps}\supseteq\set{u_k>\bar{t}-\eps}\supseteq\set{u_k>t_k},
    \]
    and so
    \begin{equation}
        \label{eq: mut-eps}
        \mu(\bar{t}+2\eps)\le m\le \mu(\bar{t}-2\eps).
    \end{equation}
    We now observe that $\mu$ is strictly decreasing in $(\essinf_\Om u,\|u\|_\infty)$. Indeed, if $\mu$ is constant in $(t_1,t_2)$ with $\essinf u<t_1<t_2<\|u\|_\infty$, then $\abs{\Set{t_1< u\le t_2}}=0$. On the other hand, since $\Om$ is connected and $u$ is continuous, $u^{-1}((t_1,t_2))$ is an open non-empty set, which has positive measure, thus a contradiction. Therefore, as by assumption $t_*\in(\essinf u,\|u\|_\infty)$,
    \eqref{eq: mut-eps} implies
    \[
        \bar{t}-2\eps\le t_*\le\bar{t}+2\eps, 
    \]
    and since $\eps$ is arbitrary, $\bar{t}=t_*$. The argument does not depend on the choice of the subsequence, and the conclusion follows.
\end{proof}
    
\begin{prop}
\label{lem: convdef}
    Let $q>n$, let $\Om\subset\R^n$ be an open bounded connected set, and let $u\in W^{3,q}(\Om)$ be a function such that for some positive constant $k$
    \[
    \abs{\na u}(x)\ge k, \qquad \qquad \forall x\in u^{-1}(0).
    \]
    Assume in addition that  $d(u^{-1}(0),\partial \Om)>0$. Let $u_j\in C^{2,\al}(\Om)$ be a sequence of functions such that 
    \[
        u_j\xrightarrow{W^{2,q}(\Om)}u.
    \]
    Then for every $j$ large enough $u_j^{-1}(0)$ is a $C^{1,s}$ hypersurface, with $s=1-n/q$, and there exist deformations $\Phi_j\in W^{2,q}(\Om;\R^n)$ such that
    \begin{enumerate}[label=(\roman*)]
        \item $\Phi_j$ are orthogonal to $u^{-1}(0)$;
        \item $\displaystyle
            u_j^{-1}(0)=(\id +\Phi_j)(u^{-1}(0));
        $
        \item $\displaystyle\lim_j\norm{\Phi_j}_{2,q}=0$
    \end{enumerate}
\end{prop} 

\begin{proof}
\begin{proofItemize}
\item
{\bf Step 1: }we first notice that for every $\eps>0$ there exists $t_0$ such that if $\abs{t}\le t_0$ then 
    \[
        u^{-1}(t)\subseteq (u^{-1}(0))^\eps,
    \]
where we use the notation 
        $(K)^t =\Set{p\in \R^n\;|\; d(p,K)\le t}$,
    for the \emph{outer parallel} set.
    
    Indeed, let us assume by contradiction that there exist ${\eps_0}>0$ and points $y_k\in u^{-1}(t_k)$ with $\lim_k t_k=0$ such that
    \[
        \forall k\in\N, \qquad d(y_k,u^{-1}(0))> \eps_0.
    \]
    As $\Om$ is bounded, up to a subsequence, we may assume that $y_k$ converge to some point $\bar{y}\in\overline{\Om}$. The continuity of $u$ gives $u(\bar{y})=0$, so that $\bar{y}\in u^{-1}(0)$. On the other hand,
    \[
    d(\bar{y},u^{-1}(0))\ge \eps_0,
    \]
    which is a contradiction.
\item {\bf Step 2: }
    We now want to extend the non-degeneracy property of the gradient of $u$ to the functions $u_j$ on their level sets $u_j^{-1}(0)$. We notice that by uniform convergence and by the previous step, for every $\eps>0$ there exist $t_0>0$ and $j_0$ such that
    \begin{equation}
    \label{eq: levelujnearlevelu}
       \forall j\ge j_0, \qquad u_j^{-1}(0)\sbs u^{-1}\ml((-t_0,t_0)\mr)\sbs (u^{-1}(0))^\eps.
    \end{equation}
    On the other hand, $\na u$ is uniformly continuous in $\Om$ 
    so there exists $\al_0>0$ such that
    \[
        \forall (x,z)\in \Om\textrm{  s.t. } |z-x|\le \al_0, \qquad\abs{\na u(z)-\na u(x)}\le \fr{k}{4}.
    \]  
In particular,
\[
\forall x\in (u^{-1}(0))^{\alpha_0}, \qquad |\nabla u(x)|\ge \frac{k}{2}.
\]
Also, $\na u_j$ uniformly converge to $\na u$, so  for $j$ large enough
    \[    
        \abs{\na u_j(x)-\na u(x)}\le \fr{k}{8} \qquad \forall x\in \Om.\]
    From the previous estimates, we get $\abs{\na u_j}(x)>0$ for every $x\in u_j^{-1}(0)$ and $j$ large enough, which implies that the sets $u_j^{-1}(0)$ are $C^{1,s}$ hypersurfaces. 
\item {\bf Step 3: }
    Let us denote $\Om_1=(u^{-1}(0))^{\alpha_0}$ and 
    \[
     \forall x\in\Om_1,\qquad    \nu(x)=-\fr{\na u}{\abs{\na u}}(x).
    \]
    We show that for every $x\in\Om_1$, the line 
    \[
    t\in[-\al_0,\al_0]\mapsto x+t\nu(x)
    \] 
    intersects $u^{-1}(0)$. Let $x\in \Om_1$. Up to choosing a smaller $\al_0$, using that $u^{-1}(0)$ is far from $\pa \Om$, we may assume that $(\Om_1)^{\al_0}\sbs \Om$. By Lagrange theorem applied to the function $\alpha\mapsto u(x+\al\nu(x))-u(x)-\al\nabla u(x)\cdot\nu(x)$, there exists $z$ on the segment between $x$ and $x+\al_0\nu(x)$ such that
    \[
        u(x+\al_0\nu(x))=u(x)-\al_0\abs{\na u}(x)+\al_0\ml(\na u(z)-\na u(x)\mr)\cdot\nu(x)
        \le u(x)-\fr{\al_0 k}{4}.
    \]
    Similarly we get
    \[
        u(x-\al_0\nu(x))\ge u(x)+\fr{\al_0 k}{4}.
    \]
    
\item {\bf Step 4:}
    Now we construct the diffeomorphisms $\Phi_j$. For every $x\in \Om_1$, let us define 
    \[  
        F_j(x,t)=u_j(x+t\nu(x))-u(x).
    \]
    We show that $F_j$ is strictly monotone in $t$ and that it always admits a zero. First we notice that
    \[
        \pa_t F_j(x,t)=\na u_j(x+t\nu(x))\cdot\nu(x).
    \]
and from Step 2 we deduce
    \[
    \forall t\in[-\alpha_0,\alpha_0],\quad \forall j\ge j_0, \qquad\pa_t F_j(x,t)\le -k/4.
    \] 
    By uniform convergence and the previous step, for $j$ large enough we have
    $F_j(x,\al_0)< 0 < F_j(x,-\al_0)$. Therefore, there exists a unique $t_j(x)\in [-\al_0,\al_0]$ such that $F_j(x,t_j(x))=0$, or, equivalently
    \begin{equation}
    \label{eq: F_j=0}
        u_j(x+t_j(x)\nu(x))=u(x).
    \end{equation}
    \item {\bf Step 5:}
    We claim that $\Phi_j(x):=t_j(x)\nu(x)$ works (up to multiplying it by a cutoff function). 

    Property \textit{(i)} follows by noticing that by construction $\nu$ is orthogonal to $u^{-1}(0)$.
    
    Let us now prove \textit{(ii)}. First \eqref{eq: F_j=0} for $x\in u^{-1}(0)$ implies
    \[
        (\id+\Phi_j)(u^{-1}(0))\subset  u_j^{-1}(0).
    \]
    For the converse inclusion, let $y\in u_j^{-1}(0)$. The point $y$ can be written as
    \[
        y=x+d(y,u^{-1}(0))\bar{\nu}(x)
    \]
    with $x\in u^{-1}(0)$ and either $\bar{\nu}=\nu(x)$ or $\bar{\nu}=-\nu(x)$. On the other hand, by \eqref{eq: levelujnearlevelu}, we have for large $j$
    \[
        u_j^{-1}(0)\sbs (u^{-1}(0))^{\al_0},
    \]
    so that $d(y,u^{-1}(0))<\al_0$. As a consequence, $d(y,u^{-1}(0))=\abs{t_j(x)}$ and we have proved that
    \[
        (\id+\Phi_j)(u^{-1}(0))= u_j^{-1}(0).
    \]

   Finally, \textit{(iii)} follows by noticing that \eqref{eq: levelujnearlevelu} implies that $t_j$ converges uniformly to $0$ and by computing
    \[
    \begin{split}
        \na t_j(x)&=-\fr{\pa_x F_j(x,t_j(x))}{\pa_t F_j(x,t_j(x))} \\
        &=\abs{\na u}(x)\fr{\na u_j(x+\Phi_j(x))-t_j(x)D\nu(x)\na u_j(x+\Phi_j(x))-\na u(x)}{\na u_j(x+\Phi_j(x))\cdot\na u(x)}.
    \end{split}
    \]
Since $W^{2,q}$ is an algebra, computing the derivatives of $t_j$ we iteratively get the convergence of higher order derivatives.
    \end{proofItemize}
\end{proof}

\newpage
\printbibliography[heading=bibintoc]

\noindent Paolo Acampora\\Dipartimento di Bioscienze e Territorio, University of Molise, and\\
Dipartimento di Matematica e Applicazioni R.Caccioppoli, University of Naples Federico II\\
paolo.acampora@unimol.it
\medskip

\noindent Jimmy Lamboley\\\'Ecole Normale Supérieure, CNRS, PSL University, DMA, and\\ Sorbonne Université, Universit\'e Paris Cit\'e, IMJ-PRG, F-75005 Paris, France\\jimmy.lamboley@ens.fr
\end{document}